\documentclass[11pt]{article}

\usepackage{enumerate,xspace}
\usepackage{amsmath,xspace,amssymb}
\usepackage[dvips]{graphics}
\usepackage{proof}
\usepackage{latexsym}  
\usepackage{euscript}  

\input xy
\xyoption{all}
\xyoption{2cell}
\UseAllTwocells
\title{Differential Equations in a Tangent Category I: \\ Complete vector fields, flows, and exponentials}
\author{J.R.B. Cockett\thanks{Partially supported by NSERC, Canada.},  
        G.S.H. Cruttwell\thanks{Partially supported by NSERC, Canada.}, and J.-S. P. Lemay\thanks{Research supported by Kellogg College, the Clarendon Fund, and the Oxford-Google DeepMind Graduate Scholarship.} 
\\ Department of Computer Science, \\
University of Calgary, Alberta, Canada and \\
Department of Mathematics and Computer Science, \\ 
Mount Allison University, Sackville, Canada, and \\
Department of Computer Science, Oxford, UK\thanks{Our thanks to Rory Lucyshyn-Wright, who provided valuable contributions to earlier versions of this work, as well as Matthew Burke, Jeff Egger, Ben MacAdam, and Bob Par\'{e} for useful discussions.}}

\newtheorem{observation}{Remark}[section]
\newtheorem{lemma}[observation]{Lemma}  
\newtheorem{theorem}[observation]{Theorem}
\newtheorem{definition}[observation]{Definition}
\newtheorem{example}[observation]{Example}
\newtheorem{remark}[observation]{Remark}

\newtheorem{proposition}[observation]{Proposition} 
\newtheorem{corollary}[observation]{Corollary} 
 
\newcommand{\x}{\times}
\newcommand{\<}{\langle}
\renewcommand{\>}{\rangle}

\newcommand{\pc}{\sigma}

\newcommand{\nc}{\eta}

\newcommand{\zdc}{0}

\newcommand{\N}{\ensuremath{\mathbb N}\xspace}
\newcommand{\R}{\ensuremath{\mathbb R}\xspace}

\newcommand{\X}{\ensuremath{\mathbb X}\xspace}

\newcommand{\evf}[1]{\ensuremath{\Delta^{#1}}}  
\newcommand{\expf}[1]{\ensuremath{E^{#1}}}  
\newcommand{\action}[1]{\ensuremath{\gamma^{#1}}}  
\newcommand{\VF}{\ensuremath{\sf{VF}(\X)}}
\newcommand{\CVF}{\ensuremath{\sf{CVF}(\X)}}
\newcommand{\FLOW}{\ensuremath{\sf{FLOW}(\X)}}

\newcommand{\smooth}{\textbf{smooth}}
\newcommand{\poly}{\textbf{poly}}
\newcommand{\sman}{\textbf{sman}}

\newcommand\nats{\hbox{$I \kern - .38em N$}} 
\newcommand\ints{\hbox{$Z \kern - .65em Z$}} 

\makeatletter

\newdimen\w@dth

\def\setw@dth#1#2{\setbox\z@\hbox{\scriptsize $#1$}\w@dth=\wd\z@
\setbox\@ne\hbox{\scriptsize $#2$}\ifnum\w@dth<\wd\@ne \w@dth=\wd\@ne \fi
\advance\w@dth by 1.2em}

\def\t@^#1_#2{\allowbreak\def\n@one{#1}\def\n@two{#2}\mathrel
{\setw@dth{#1}{#2}
\mathop{\hbox to \w@dth{\rightarrowfill}}\limits
\ifx\n@one\empty\else ^{\box\z@}\fi
\ifx\n@two\empty\else _{\box\@ne}\fi}}
\def\t@@^#1{\@ifnextchar_ {\t@^{#1}}{\t@^{#1}_{}}}

\def\t@left^#1_#2{\def\n@one{#1}\def\n@two{#2}\mathrel{\setw@dth{#1}{#2}
\mathop{\hbox to \w@dth{\leftarrowfill}}\limits
\ifx\n@one\empty\else ^{\box\z@}\fi
\ifx\n@two\empty\else _{\box\@ne}\fi}}
\def\t@@left^#1{\@ifnextchar_ {\t@left^{#1}}{\t@left^{#1}_{}}}

\def\two@^#1_#2{\def\n@one{#1}\def\n@two{#2}\mathrel{\setw@dth{#1}{#2}
\mathop{\vcenter{\hbox to \w@dth{\rightarrowfill}\kern-1.7ex
                 \hbox to \w@dth{\rightarrowfill}}%
       }\limits
\ifx\n@one\empty\else ^{\box\z@}\fi
\ifx\n@two\empty\else _{\box\@ne}\fi}}
\def\tw@@^#1{\@ifnextchar_ {\two@^{#1}}{\two@^{#1}_{}}}

\def\tofr@^#1_#2{\def\n@one{#1}\def\n@two{#2}\mathrel{\setw@dth{#1}{#2}
\mathop{\vcenter{\hbox to \w@dth{\rightarrowfill}\kern-1.7ex
                 \hbox to \w@dth{\leftarrowfill}}%
       }\limits
\ifx\n@one\empty\else ^{\box\z@}\fi
\ifx\n@two\empty\else _{\box\@ne}\fi}}
\def\t@fr@^#1{\@ifnextchar_ {\tofr@^{#1}}{\tofr@^{#1}_{}}}

\newdimen\W@dth
\def\setW@dth#1#2{\setbox\z@\hbox{$#1$}\W@dth=\wd\z@
\setbox\@ne\hbox{$#2$}\ifnum\W@dth<\wd\@ne \W@dth=\wd\@ne \fi
\advance\W@dth by 1.2em}

\def\T@^#1_#2{\allowbreak\def\N@one{#1}\def\N@two{#2}\mathrel
{\setW@dth{#1}{#2}
\mathop{\hbox to \W@dth{\rightarrowfill}}\limits
\ifx\N@one\empty\else ^{\box\z@}\fi
\ifx\N@two\empty\else _{\box\@ne}\fi}}
\def\T@@^#1{\@ifnextchar_ {\T@^{#1}}{\T@^{#1}_{}}}

\def\T@left^#1_#2{\def\N@one{#1}\def\N@two{#2}\mathrel{\setW@dth{#1}{#2}
\mathop{\hbox to \W@dth{\leftarrowfill}}\limits
\ifx\N@one\empty\else ^{\box\z@}\fi
\ifx\N@two\empty\else _{\box\@ne}\fi}}
\def\T@@left^#1{\@ifnextchar_ {\T@left^{#1}}{\T@left^{#1}_{}}}

\def\Tofr@^#1_#2{\def\N@one{#1}\def\N@two{#2}\mathrel{\setW@dth{#1}{#2}
\mathop{\vcenter{\hbox to \W@dth{\rightarrowfill}\kern-1.7ex
                 \hbox to \W@dth{\leftarrowfill}}%
       }\limits
\ifx\N@one\empty\else ^{\box\z@}\fi
\ifx\N@two\empty\else _{\box\@ne}\fi}}
\def\T@fr@^#1{\@ifnextchar_ {\Tofr@^{#1}}{\Tofr@^{#1}_{}}}

\def\Two@^#1_#2{\def\N@one{#1}\def\N@two{#2}\mathrel{\setW@dth{#1}{#2}
\mathop{\vcenter{\hbox to \W@dth{\rightarrowfill}\kern-1.7ex
                 \hbox to \W@dth{\rightarrowfill}}%
       }\limits
\ifx\N@one\empty\else ^{\box\z@}\fi
\ifx\N@two\empty\else _{\box\@ne}\fi}}
\def\Tw@@^#1{\@ifnextchar_ {\Two@^{#1}}{\Two@^{#1}_{}}}

\def\to{\@ifnextchar^ {\t@@}{\t@@^{}}}
\def\from{\@ifnextchar^ {\t@@left}{\t@@left^{}}}
\def\tofro{\@ifnextchar^ {\t@fr@}{\t@fr@^{}}}
\def\To{\@ifnextchar^ {\T@@}{\T@@^{}}}
\def\From{\@ifnextchar^ {\T@@left}{\T@@left^{}}}
\def\Two{\@ifnextchar^ {\Tw@@}{\Tw@@^{}}}
\def\Tofro{\@ifnextchar^ {\T@fr@}{\T@fr@^{}}}

\makeatother

\input{diagxy}

\begin{document}
\maketitle

\begin{abstract}
This paper describes how to define and work with differential equations in the abstract setting of tangent categories. The key notion is that of a \emph{curve object} which is, for differential geometry, the structural analogue of a natural number object. A curve object is a preinitial object for dynamical systems; dynamical systems may, in turn, be viewed as determining systems of differential equations. The unique map from the curve object to a dynamical system is a solution of the system, and a dynamical system is said to be complete when for all initial conditions there is a solution. A subtle issue concerns the question of when a dynamical system is complete, and the paper provides abstract conditions for this.
 
This abstract formulation also allows new perspectives on topics such as commutative vector fields and flows. In addition, the stronger notion of a {\em differential curve object}, which is the centrepiece of the last section of the paper, has exponential maps and forms a differential exponential rig. This rig then, somewhat surprisingly, has an action on every differential object and bundle in the setting. In this manner, in a very strong sense, such a curve object plays the role of the real numbers in standard differential geometry.

\end{abstract}


\newpage

\tableofcontents 


\section{Introduction}

The ability to solve ordinary differential equations plays an important role in many aspects of differential geometry: for example, in finding geodesics or parallel transport of a connection. Thus, in the setting of tangent categories - which aim to be an axiomatic categorical setting for differential geometry - it is important to be able to describe differential equations and their solutions.  
While many ideas from differential geometry have been generalized to tangent categories (for example, vector bundles \cite{diffBundles}, connections \cite{connections}, and differential forms \cite{diffForms}), one cannot solve ordinary differential equations or turn vector fields into flows in an arbitrary tangent category. However, it is possible to characterize the structure which enables such solutions, and this is the purpose of the paper.

The central idea is that of a \textbf{curve object}: an object which allows unique solutions of ordinary differential equations. A curve object is analogous to a natural numbers object, both in form and philosophy. Its definition (Definition \ref{defnCurveObject}) resembles that of a (parametrized) natural numbers object, and possesses a similarly powerful universal property which leads to a rich theory. Just as a natural number object acquires a commutative monoid structure, so too does a curve object (Theorem \ref{thmCurveMonoidStructure}). However, there is a key difference between curve objects and natural numbers objects: while natural number objects are initial, so that {\em every\/} system has a unique solution, curve objects are only preinitial, so that solutions are unique but need not exist.  When some special systems (e.g. linear systems) do have solutions, this leads to far-reaching consequences.    

When a tangent category has a curve object, one can show that many standard results from differential geometry hold. There is a bijection between complete vector fields and flows (Theorem \ref{thmCVFFlow}). There is a correspondence between commuting vector fields and commuting flows (Theorem \ref{thmCommFlows}). The sum of commuting complete vector fields is complete (Proposition \ref{propSumVFs}). Higher-order differential equations can be reduced to first-order ones (Corollary \ref{corHigherSolutions}). One can define geodesics (Section \ref{secGeodesics}) and parallel transport (Remark \ref{rmkParallelTransport}) for connections. Finally, and most significantly (see more below) when the curve object is a \emph{differential} curve object for which linear systems exist, exponential functions exist (Definition \ref{defnExpFunctions}) and can be used to define bilinear actions of the curve object on any differential bundle in the tangent category.  

In addition to being able to prove basic results, the expression of these ideas in the abstract setting of a tangent category with a curve object also leads to new perspectives. An example of this is the relationship between commuting vector fields and commuting flows. It is a standard result in differential geometry that the flows of two vector fields commute (ie., can be applied independently of order) if and only if the Lie bracket of the vector fields is 0. Our proof of this result in the setting of tangent categories is very different from the standard proofs. Rather than a calculational proof, we provide a structural proof. In particular, we show that a pair of commuting vector fields is a vector field in the tangent category of vector fields; similarly, a pair of commuting flows is a flow in the tangent category of flows. As a result, the correspondence between commuting vector fields and commuting flows follows almost immediately: see Section \ref{secCommFlowsAndVFs}.  

The results above can be obtained from very few assumptions. However, it is natural to assume that the curve object is a differential object and that every linear system should have a solution; these assumptions lead to further surprising structure. In a tangent category, there is no assumption of any basic ``real line'' object. In particular, the analog of vector fields (differential objects) and vector bundles (differential bundles) assume no action of any sort of real number object. Nevertheless, we show that if the tangent category has a differential curve object $C$ (see Definition \ref{defnDiffCurveObject}) and solves linear systems, then $C$ acquires the structure of a differential exponential rig (Definition \ref{defnDiffExponential}) and every differential object and differential bundle acquires the structure of a $C$-module. Furthermore, linearity becomes equivalent to preserving $C$-module structure.  

To give the definition of a curve object, it is helpful to first develop a number of preliminary ideas. Thus, the paper is organized as follows: in Section \ref{secVFs}, we first focus on vector fields and dynamical systems in a tangent category. These ideas can be defined for any tangent category, and already lead to some interesting perspectives (for example, as mentioned above, that a vector field in the tangent category of vector fields is a pair of commuting vector fields: Proposition \ref{propCommuteVFs}). In Section \ref{secSolutions}, we define what it means to give a solution to a dynamical system, and develop some initial results regarding such solutions. In Section \ref{secCurveAndFlows}, we define the notion of a curve object and prove many of the key results mentioned above. In Section \ref{secLinearSystems}, we investigate the consequences of combining differential and curve object structure in the presence of linear completeness. In particular, we establish that a differential curve object is a differential exponential rig which has a bilinear action on every differential object and differential bundle in the tangent category.

 
\section{Vector fields and dynamical systems}\label{secVFs}

\subsection{Preliminaries on tangent categories}\label{secPrelims}



We assume the reader is familiar with the basic theory of tangent categories as presented in \cite{sman3}. As in that paper, we write composition in diagrammatic order, so that $f$, followed by $g$, is written as $fg$. In this section, we add a few ideas to the basic theory of tangent categories that will be relevant for this paper, and make some comments about notation. First, recall the following:

\begin{definition}
A \textbf{Cartesian tangent category} is a tangent category with all finite products (including a terminal object $1$) and such that $T$ preserves finite products. \end{definition}

We write $\pi_0, \pi_1$ for the projections out a product, $\<f,g\>$ for the induced unique map into a product, and for $f: A \to A', g: B \to B'$, $f \times g$ for the pairing $\<\pi_0f,\pi_1g\>: A \times A' \to B \times B'$.

A Cartesian category $\X$ allows one to build the simple fibration $\X[\X]$ over it (eg., see \cite[pg. 40]{bart-jacobs}). If $\X$ has tangent structure, this fibration is also a tangent fibration, with each of its fibres (``morphisms in context'') also having tangent structure. This structure will be helpful later when we consider ``dynamical systems in context'' (see section \ref{secDynSystems}).  

\begin{definition}
If $\X$ is a Cartesian category, then the \textbf{simple fibration over \X} is a category with:
\begin{itemize}
	\item Objects: pairs $(X,\Gamma)$ where $X, \Gamma$ are objects of $\X$ ($X$ is the object in the fibre while $\Gamma$ is the context);
	\item Maps: $(h,f): (X,\Gamma) \to (X',\Gamma')$ where $h: X \x \Gamma \to X'$ and $f: \Gamma \to \Gamma'$;
	\item Composition: $(h,f)(k,g) := (\<h,\pi_1 f\>k,fg)$ and identities: $1_{(X,\Gamma)} := (\pi_0,1_\Gamma): (X,\Gamma) \to (X,\Gamma)$.
\end{itemize}
\end{definition}

\begin{proposition}
For any Cartesian tangent category $\X$ its simple fibration $\X[\X]$ is a tangent fibration (in the sense of \cite[Definition 5.2]{diffBundles}) so that each fibre is a tangent category, the substitution functors preserve the tangent structure, and the total category is a tangent category.
\end{proposition}
\begin{proof}
This is a straightfoward exercise, with $T(X,\Gamma) := (TX,T\Gamma)$ and $T(h,f) = (Th,Tf)$ (where we identify $T(X) \times T(\Gamma)$ with $T(X \times \Gamma))$.
\end{proof}

\begin{remark}\label{rmkContext}
For any fixed object (``context'') $\Gamma$, one then gets a tangent category $\X[\Gamma]$ (as in \cite[Theorem 5.3]{diffBundles}), consisting of:
\begin{itemize}
	\item Objects: as in $\X$;
	\item Maps: $h: X \to Y$ in $\X[\Gamma]$ is a map $h: X \times \Gamma \to Y$ in $\X$;
	\item Composition: for $h: X \to Y$, $k: Y \to Z$ in $\X[\Gamma]$, $hk: X \to Z$ is the composite
		\[ X \times \Gamma \to^{\<h,\pi_1\>} Y \times \Gamma \to^{k} Z. \]
	\item Tangent functor: on objects, apply $T$; given a map $h: X \to Y$ in $\X[\Gamma]$, $T(h)$ is the composite
		\[ TX \times \Gamma \to^{1 \times 0} T(X \times \Gamma) \to^{T(h)} TY \]
	(where $(1 \x 0): T(X) \x \Gamma \to T(V \x \Gamma)$ is the ``strength'' of the tangent functor).
	\item Natural transformations: each without reference to context; for example, the projection is
		\[ TM \times \Gamma \to^{\pi_0} TM \to^{p_M} M. \]
\end{itemize}
\end{remark}

At several places in this paper, we will make use of the ``bracket'' operation $\{ \}$ described in \cite[Lemma 2.12.i]{sman3} which exists in any tangent category. This is an operation which takes a map $f: V \to T^2M$ such that $fT(p) = fpp0$, and produces a map $\{f\}: V \to TM$.   The properties of the bracket operation can be found in \cite[Lemma 2.14]{sman3}. Here we add one more basic property which was not proved there.

\begin{lemma}\label{lemmaBracketZero}
For an object $M$ in a tangent category, $\{0_{TM}\} = p_M0_M$.  
\end{lemma}
\begin{proof}
Suppressing subscripts, we need to show that $p0$ satisfies the same universal property as $\{0\}$; that is, we need to show that $\<p0\ell, 0p0\>T(+) = 0$. Indeed, using naturality of $0$,
	\[ \<p0\ell, 0p0\>T(+) = \<p00, 0\>T(+) = \<p0,1\>+0 = 0, \]
as required.  
\end{proof}

We will occasionally be working with differential bundles $(q: A \to M, \lambda, \sigma, \zeta)$, as first defined in \cite{diffBundles}. For such bundles, the map $\mu: A_2 \to TA$ is of fundamental importance, and will play a key role in defining exponential functions for these bundles. However, in this paper, we will use a slightly different convention for $\mu$ than in the papers \cite{diffBundles, connections}. There, the map $\mu$ was defined as the composite
	\[ A_2 \to^{\<\pi_0 \lambda, \pi_1 0\> = (\lambda \times 0)} T(A_2) \to^{T(\sigma)} TA. \]
In this paper, however, we will switch the roles of the $\lambda$ and the $0$ terms; that is, in this paper, we will define $\mu$ as the composite
	\[ A_2 \to^{\<\pi_0 0, \pi_1 \lambda\> = (0 \times \lambda)} T(A_2) \to^{T(\sigma)} TA. \]
Thus, in local co-ordinates in smooth manifolds, if $(x,v_1,v_2)$ is an element of $A_2$ (with $x$ the point and $v_i$ the vectors), then $\mu$ has the effect
	\[ (x,v_1,v_2) \mapsto (x,v_1,0,v_2). \]
This is more consistent with standard differential geometry terminology than the original definition. For the most part, this change is relatively minor; the only real effect is to change the roles of the two projections out of $A_2$ for results relating to $\mu$. For example, \cite[Lemma 2.9]{diffBundles} showed that $\mu p = \pi_1$; with this change, the result is now $\mu p = \pi_0$ instead.  

Throughout the rest of this paper, we work in a given tangent category $\X$.

\subsection{Vector fields and their morphisms}

A central object of study in this paper is the notion of a vector field. 

\begin{definition}
A \textbf{vector field} on an object M consists of a map $V: M \to TM$ which is a section of $p_M: M \to TM$; that is, $Vp_M = 1_M$.
\end{definition}

\begin{example}\label{exVecFields}
\begin{enumerate}[(i)]
	\item If $A$ is a differential object (so that there is a map $\hat{p}: TA \to A$ so that $(TA,p_A,\hat{p})$ is a product diagram; see \cite[Definition 3.1]{sman3}) then a vector field $V$ on $A$ is equivalent to simply giving a map $\hat{V} := V\hat{p}: A \to A$, since the first component of the map $V: A \to TA \cong A \times A$ must be the identity.
	\item In a representable tangent category (where the functor $T$ is represented by an object $D$ (see \cite[Section 5.2]{sman3}), a vector field can be transposed into an action:
    		\[ \infer={D \x M \to_{\widetilde{V}} M}{ M \to^V T(M) = M^D} \]
	thus in this case we may view a vector field as an infinitesimal state transition.
\end{enumerate}
\end{example}

There is an obvious notion of a morphism of vector fields.

\begin{definition}
If $V_1$ is a vector field on $M_1$ and $V_2$ a vector field on $M_2$, then a \textbf{vector field morphism} from $(M_1,V_1)$ to $(M_2,V_2)$ consists of a map $f: M_1 \to M_2$ so that the following diagram commutes:
\[
\bfig
	\square<500,350>[M_1`TM_1`M_2`TM_2;V_1`f`T(f)`V_2]
\efig
\]
\end{definition}
In differential geometry, in such a situation one says that $V_1$ and $V_2$ are ``$f$-related'' (for example, see \cite[pg. 87]{lee}). An important property of such maps is how they interact with the Lie bracket (for this result in standard differential geometry, see \cite[Prop. 4.16]{lee}):

\begin{proposition}
Suppose the tangent category $\X$ has negatives\footnote{That is, the addition $+$ is invertible in each slice; see \cite[Section 3.3]{sman3}.}, $V_1$ and $V_2$ are vector fields on $M$, and $W_1$ and $W_2$ are vector fields on $M'$. Then if $f$ is a vector field morphism from $(M,V_1)$ to $(M',W_1)$ and from $(M,V_2)$ to $(M',W_2)$, then $f$ is also a vector field morphism from $(M,[V_1,V_2])$ to $(M',[W_1,W_2])$.  
\end{proposition}
\begin{proof}
We want to show that the following diagram commutes:
\[
\bfig
	\square<650,350>[M`TM`M'`TM';\mbox{$[V_1,V_2]$}`f`T(f)`\mbox{$[W_1,W_2]$}]
\efig
\]
Indeed, consider
\begin{eqnarray*}
&     & [V_1,V_2]T(f) \\
& = & \{V_1T(V_2) - V_2T(V_1)c\}T(f) \\
& = & \{V_1T(V_2)T^2(f) - V_2T(V_1)cT^2(f)\} \mbox{ (by \cite[Lemma 2.14.vii]{sman3})} \\
& = & \{V_1T(V_2T(f)) - V_2T(V_1)T^2(f)c \} \mbox{ (using naturality of $c$)} \\
& = & \{V_1T(fW_2) - V_2T(V_1T(f))c \} \mbox{ (since $f: (M,V_2) \to (M',W_2)$)} \\
& = & \{V_1T(f)T(W_2) - V_2T(fW_1)c \} \mbox{ (since $f: (M,V_1) \to (M',W_1)$)} \\
& = & \{fW_1T(W_2) - V_2T(f)T(W_1)c \} \mbox{ (since $f: (M,V_1) \to (M',W_1)$)} \\
& = & \{fW_1T(W_2) - fW_2T(W_1)c \} \mbox{ (since $f: (M,V_2) \to (M',W_2)$)} \\
& = & f \{W_1T(W_2) - W_2T(W_1)c \} \mbox{ (by \cite[Lemma 2.14.i]{sman3})} \\
& = & f[W_1,W_2]
\end{eqnarray*}
as required.  
\end{proof}

Vector fields and their morphisms clearly form a category. However, even better, they also form a tangent category.

\begin{proposition}\label{propVFTanCat}
Given a tangent category $\X$:
\begin{enumerate}[(i)]
	\item There is a tangent category, $\VF$, whose objects are pairs $(M,V)$, with $V$ a vector field on $M$, whose morphisms are vector field morphisms, whose tangent functor $\bar{T}$ is 
	\[ \bar{T}(M,V) := (TM,T(V)c), \bar{T}(f) := T(f), \]
and whose structural transformations are as in $\X$. 
	\item If $U: \VF \to \X$ is the forgetful functor, then $U$ is a tangent functor.  
	\item If $\X$ is a Cartesian tangent category, then $\VF$ is also, with terminal object $(1,0)$ and pointwise product
	\[ (M_1,V_1) \times (M_2,V_2) := (M_1 \times M_2, V_1 \times V_2). \]
\end{enumerate}
\end{proposition}
\begin{proof}
The majority of the proof is left to the reader, as one can almost entirely follow the proof of a similar result (about tangent categories of connections) in \cite{affine}: see 5.1 and 5.11--5.16 in that paper. The idea is to first prove that $\VF$ is part of a 2-functor from the 2-category $\textbf{tan}$ of tangent categories to the 2-category of categories; as the structural functors and transformations of a tangent category live in $\textbf{tan}$ this immediately gives that $\VF$ has all the structural components of a tangent category. The remaining question is then about the required limits in this category.  Here we do one part of the proof of this result: the creation of ``tangent'' limits by the forgetful functor from $\VF$ to $\X$. That is, suppose that $D: {\mathbb J} \to \VF$ is a functor, and $\pi = \{\pi_j: L \to D_J\}_{j \in {\mathbb J}}$ is a cone on $D$. Suppose that $U$ sends $\pi$ to a limit cone for $UD$ that is preserved by $T^k$ for each $k \in \N$. Then we claim that $\pi$ is a limit cone for $D$.

Write $D_j = (UD_j, V_j)$ and $L = (UL,V_L)$. Let $\{z_j: (Z,V) \to D_j\}_{j \in {\mathbb J}}$ be a cone on $D$. Hence $\{z_j: Z \to UD_j\}_{j \in {\mathbb J}}$ is a cone on $UD$, and so by assumption induces a morphism $z: Z \to UL$ in $\X$. We need this to be a vector field morphism. For each $j \in {\cal J}$ we can calculate
	\[ VT(z)T(\pi_j) = VT(z\pi_j) = VT(z_j) = z_jV_j = Z\pi_j V_j = zV_LT(\pi_j) \]
since each $z_j$ and $\pi_j$ is a vector field morphism. Hence by the universal property of the limit cone $T(\pi)$, $VT(z) = z V_L$ and so $z$ is indeed a vector field morphism. This proves the claim. The rest of the proof follows as in 5.11--5.16 in \cite{affine}.  
\end{proof}

\subsection{Commuting vector fields}

It is important to be able to determine when a pair of vector fields ``commute'': when they do, their associated flows can be applied independently of order.     Here, we begin with a definition of commutation which works in any tangent category, before showing its equivalence to the standard definition in smooth manifolds (the Lie bracket of the vector fields being 0).  We shall also give a new alternative characterization of this definition, by considering vector fields in the tangent category of vector fields (Proposition \ref{propCommuteVFs}).  

\begin{definition}\label{defnCommuteVFs}
If $V_1$ and $V_2$ are vector fields on an object $M$, then we say that $V_1$ and $V_2$ \textbf{commute} if $V_1T(V_2)c = V_2T(V_1)$.
\end{definition}

\begin{example}
For any vector field $V$ on $M$, $0_M$ commutes with $V$, as 
	\[ VT(0_M)c = VT(0_M) = 0_MV \]
by coherence of $c$ and naturality of $0$.
\end{example}

For smooth manifolds, the standard definition of vector fields commuting asks that their bracket be 0, ie., $[V_1,V_2] = 0$.  In tangent categories, however, the Lie bracket only exists if the tangent category has negatives.  In this case, however, the two definitions agree (and so, in particular, the definitions agree in smooth manifolds):

\begin{lemma}\label{lemmaCommFlows4}
Suppose $\X$ has negatives, and $V_1$ and $V_2$ are vector fields on an object $M$.  Then $V_1T(V_2)c = V_2T(V_1)$ if and only if $[V_1,V_2] = 0$.  
\end{lemma}
\begin{proof}
Recall (see \cite[Section 3.4]{sman3} and \cite{lieBracket}) that in a tangent category with negatives, 
	\[ [V_1,V_2] := \{V_1T(V_2) - V_2T(V_1)c\}. \]
First, suppose that $V_1T(V_2)c = V_2T(V_1)$.  Since $c^2 = 1$, this implies $V_1T(V_2) = V_2T(V_1)c$.  Then using Lemma \ref{lemmaBracketZero},
	\[ [V_1,V_2] = \{V_1T(V_2) - V_2T(V_1)c\} = \{V_1T(V_2)p0\} = \{V_20\} = V_2\{0\} = V_2p0 = 0. \] 

Now suppose that $[V_1,V_2] = 0$; that is, $\{V_1T(V_2) - V_2T(V_1)c\} = 0$.  Thus, by universality of the bracket operation,
\[ 	\<0\ell, (V_1T(V_2) - V_2T(V_1)c)p0\>T(+) = V_1T(V_2) - V_2T(V_1)c \]
But we have
\begin{eqnarray*}
&   & \<0\ell, (V_1T(V_2) - V_2T(V_1)c)p0\>T(+) \\
& = & \<0T(0), V_20\>T(+) \\
& = & 0V_2
\end{eqnarray*}
Thus $ V_1T(V_2) - V_2T(V_1)c = 0V_2$; adding $V_2T(V_1)c$ to both sides gives $V_1T(V_2) = V_2T(V_1)c$.  Since $c^2 = 1$, this implies $V_1T(V_2)c = V_2T(V_1)$, as required.  
\end{proof}

The commutation relation is symmetric:
\begin{lemma}\label{lemmaCommSymmetric}
For vector fields $V_1, V_2$ on $M$, $V_1$ commutes with $V_2$ if and only if $V_2$ commutes with $V_1$.
\end{lemma}
\begin{proof}
Since $c^2 = 1$, $V_1T(V_2)c = V_2T(V_1)$ is equivalent to $V_2T(V_1) = V_1T(V_2)c$.  
\end{proof}

There is no reason why this relation should be reflexive in an arbitrary tangent category.  However, at this point, we do not have a specific example of a vector field which does not commute with itself.  In fact, as the next few results show, in all standard tangent categories the relation is reflexive.  

\begin{lemma}
Suppose that $M$ is an object whose tangent bundle can be equipped with a torsion-free effective vertical connection (\cite[Definition 6.6]{roryConnections}).  Then for any vector field $V$ on $M$, $V$ commutes with itself.  
\end{lemma}
\begin{proof}
Since $K$ is an effective vertical connection, by definition, $T^2M$ is a pullback of three copies of $TM$, with projections $K$, $T(p)$, and $p$.  Thus, since $VT(V)c$ and $VT(V)$ are two maps with codomain $T^2M$, it suffices to check that the two maps are equal when post-composed by $K$, $T(p)$, and $p$.  Since $K$ is torsion-free, $cK=K$, and so $VT(V)cK = VT(V)K$.  The other two equalities are true of any vector field:
	\[ VT(V)cp = VT(V)T(p) = VT(Vp)= V = VpV = VT(V)p \]
and
	\[ VT(V)cT(p) = VT(V)p = VpV = V = VT(Vp) = VT(V)T(p). \] 
\end{proof}
In the category of smooth manifolds, every object can be equipped with a torsion-free effective connection, and so the above result applies to every object in the tangent category of smooth manifolds.  We also have the following:
\begin{corollary}\label{corFlipCondition}
If $M$ is an object which can be equipped with differential structure, then for any vector field $V$ on $M$, $V$ commutes with itself.  
\end{corollary}
\begin{proof}
By example 3.25 in \cite{connections}, a differential object has a torsion-free vertical connection on its tangent bundle; it is easy to check that this connection is effective.  Thus, by the previous result, for any vector field $V$ on $M$, $VT(V)c = VT(V)$.
\end{proof}

Another assumption which makes self-commutation automatic is a global assumption on the tangent category itself:
\begin{lemma}
Suppose $\X$ is a tangent category with negatives in which for any $x: M \to TM$, $\<x,x\>+ = 0$ implies $x=0$.  Then for any vector field $V$ on $M$, $V$ commutes with itself.  
\end{lemma}
\begin{proof}
By Theorem 3.17 in \cite{sman3}, $[V,V] = [V,V]-$; hence $\<[V,V],[V,V]\>+ = 0$ and so by the assumption $[V,V]=0$. By Lemma \ref{lemmaCommFlows4}, this implies that $VT(V)c = VT(V)$.
\end{proof}

Returning to the general theory of commuting vector fields, there is a completely different way to view them using the tangent category $\VF$:

\begin{proposition}\label{propCommuteVFs}
A pair of commuting vector fields is equivalent to a vector field in the tangent category $\VF$.  
\end{proposition}
\begin{proof}
A vector field in $\VF$ consists of an object $(M,V)$ in $\VF$, together with a map $Y: (M,V) \to \bar{T}(M,V)$ in $\VF$ such that $Yp = 1_M$.  Thus in particular $Y$ is a vector field on $M$ in $V$.  However, the fact that $Y$ is a vector field morphism from $(M,V)$ to $\bar{T}(M,V) = (TM,T(V)c)$ means that the following diagram must also commute:
\[
\bfig
	\square<500,350>[M`TM`TM`T^2M;V`Y`T(Y)`T(V)c]
\efig
\]
in other words, the vector fields $V$ and $Y$ commute.  Conversely, given a pair of commuting vector fields $V$ and $Y$, the above also shows that $Y$ is a vector field on $(M,V)$ in $\VF$.
\end{proof}

There is also a characterization in terms of a single map from $M$ to its second-order tangent bundle $T^2M$:
\begin{proposition}
The following are equivalent:
\begin{enumerate}[(i)]
	\item a vector field $V_1: (M,V_2) \to \bar{T}(M,V_2)$ in $\VF$;		
	\item a pair of commuting vector fields $V_1, V_2$;	
	\item a map $z: M \to T^2M$ such that $zpp = 1$ and $zT(p)T(z)T(p) = zpT(z)T^2(p)c$.
\end{enumerate}
\end{proposition}
\begin{proof}
The equivalence of (i) and (ii) was established in the previous result.  To prove (ii) implies (iii), define $z := V_1T(V_2)$.  Then
	\[ zT(p)T(z)T(p) = V_1T(V_2)T(p)T(V_1T(V_2))T(p) = V_1T(V_1T(V_2)p) = V_1T(V_1pV_2) = V_1T(V_2) \]
while
	\[ zpT(z)T^2(p)c = V_1T(V_2)p T(V_1T(V_2))T^2(p)c = V_1p V_2 T(V_1T(v_2)T(p))c = V_2T(V_1)c \]
so the two are equal since $V_1$ and $V_2$ commute.

Finally, we will prove (iii) implies (ii): given such a $z$, define $V_1 := zT(p)$ and $V_2 := zp$.  Then
	\[ V_1T(V_2) = zT(p)T(zp) = zT(p)T(z)T(p) = zpT(z)T^2(p)c = V_2T(V_1)c \]
so that $V_1$ and $V_2$ commute.
\end{proof}

\subsection{Linear vector fields}

Linear vector fields will play a key role in later sections of this paper. Here, we first define them in generality (on differential bundles), then specify to differential objects:

\begin{definition}\label{defnLinearLifting}
Suppose that $q: A \to M$ is a differential bundle (with lift map $\lambda: A \to TA$). A \textbf{linear vector field} on $q: A \to M$ consists of a pair of vector fields $V^A: A \to TA, V^M: M \to TM$ so that $q$ is a vector field morphism between them (that is, the following diagram commutes):
\[
\bfig
	\square<500,350>[A`TA`M`TM;V^A`q`T(q)`V^M]
\efig
\]
and so that $(V^A,V^M)$ is a linear map (see \cite[Definition 2.3]{diffBundles}) from $q: A \to TA$ to $T(q): TA \to T^2(A)$, that is the following diagram commutes:
\[
\bfig
	\square<500,350>[A`TA`TA`T^2A;V^A`\lambda`T(\lambda)c`T(V^A)]
\efig
\]
In this situation, we say that $V^A$ is \textbf{over} $V^M$.  
\end{definition}
A linear vector field can also be viewed as a vector field in the tangent category of differential bundles and linear morphisms between them.  

Linear vector fields on differential objects have a simpler form:

\begin{example}\label{rmkLinVFs}
If $A$ is a differential object, then in particular it is a differential bundle over $1$. A vector field on $1$ is simply the identity, so a linear vector field on $A$ consists of just a linear vector field $V: A \to TA$; as in Example \ref{exVecFields}(ii), this is equivalent to simply giving a linear map $\hat{V} := V\hat{p}: A \to A$.
\end{example}

Every linear vector field in $\X$ gives rise to a differential bundle in $\VF$:

\begin{proposition}\label{propDiffBundlesInVF}
If $q: A \to M$ is a differential bundle (with lift map $\lambda: A \to TA$) and $(V^A,V^M)$ is a linear vector field on it, then $(q: (A,V^A) \to (M,V^M),\lambda)$ is a differential bundle in $\VF$.
\end{proposition}
\begin{proof}
By definition, $q$ is a vector field morphism from $V^A$ to $V^M$. We also need to show that $\lambda$ is a vector field morphism from $(A,V^A)$ to $\bar{T}(A,V^A)$. However, the diagram for $\lambda$ to be a vector field morphism asks that
\[
\bfig
	\square<500,350>[A`TA`TA`T^2A;V^A`\lambda`T(\lambda)c`T(V^E)]
\efig
\]
commutes, which is the same as asking that $V^A$ be linear. The universality of $\lambda$ in $\VF$ then follows since the forgetful functor to $\X$ creates tangent limits (see the proof of Proposition \ref{propVFTanCat}).  
\end{proof}
We shall see later (Proposition \ref{propEndemicDiffBundlesInVF}) that by restricting attention to ``endemic'' differential bundles, the above essentially characterizes differential bundles in $\VF$.  

For linear vector fields on differential objects, the commutation condition can be simplified:

\begin{proposition}
Suppose that $A$ is a differential object (with associated projection $\hat{p}: TA \to A$) and $V_1,V_2: A \to TA$ are linear vector fields on it, with corresponding linear endomorphisms $\hat{V_1} := V_1 \hat{p}, \hat{V_2} := V_2\hat{p}$ (see Example \ref{rmkLinVFs}). Then $V_1$ and $V_2$ commute if and only if $\hat{V_1}\hat{V_2} = \hat{V_2}\hat{V_1}$.
\end{proposition}
\begin{proof}
First, note that for any vector fields $V_1$ and $V_2$, $V_1T(V_2)c$ and $V_2T(V_1)c$ are equal when post-composed by $T(p)$ and $p$:
	\[ V_1T(V_2)cT(p) = V_1T(V_2)p = V_1pV_2 = V_2 \mbox{ while } V_2T(V_1)T(p) = V_2 T(V_1p) = V_2, \]
and
	\[ V_1T(V_2)cp = V_1T(V_2)T(p) = V_1T(V_2p) = V_1 \mbox{ while } V_2T(V_1)p = V_2pV_1 = V_1. \]

Now, the maps $V_1T(V_2)c$ and $V_2T(V_1)$ are maps to $T^2A$. Since $TA$ is a product with projections $(p,\hat{p})$, and $T$ preserves this product, $T^2A$ is also a product, with projections $T(p)$ and $T(\hat{p})$. The above shows that $V_1T(V_2)$ and $V_2T(V_1)c$ are equal when post-composed by $T(p)$, so they are equal if and only if they are equal when post-composed by $T(\hat{p})$. Now $V_1T(V_2)cT(\hat{p})$ and $V_2T(V_1)\hat{p}$ are maps into $TA$, so they are equal if and only if they are equal when post-composed by $p$ and $\hat{p}$. However, by naturality of $p$, $T(\hat{p})p = p \hat{p}$, and by above, $V_1T(V_2)c$ and $V_2T(V_1)$ are equal when post-composed by $p$. Thus $V_1T(V_2)c = V_2T(V_1)$ if and only if 
	\[ V_1T(V_2)cT(\hat{p})\hat{p} = V_2T(V_1)T(\hat{p})\hat{p}. \]

However, we can simplify both sides of this equation further. Indeed,
\begin{eqnarray*}
&   & V_1T(V_2)cT(\hat{p})\hat{p} \\
& = & V_1T(V_2)T(\hat{p})\hat{p} \mbox{ (by \cite[Proposition 3.6]{diffBundles})} \\
& = & V_1T(\hat{V_2})\hat{p} \mbox{ (by definition of $\hat{V_2}$)} \\
& = & V_1\hat{p}\hat{V_2} \mbox{ (by linearity of $\hat{V_2}$)} \\
& = & \hat{V_1}\hat{V_2} \mbox{ (by definition of $\hat{V_1}$)} 
\end{eqnarray*}
while a similar calculation shows that $V_2T(V_1)T(\hat{p})\hat{p} = \hat{V_2}\hat{V_1}$.  Thus $V_1$ and $V_2$ commute if and only if $\hat{V_1}\hat{V_2} = \hat{V_2}\hat{V_1}$.
\end{proof}
 This is a generalization of the result in differential geometry that linear vector fields commute if and only if their associated matrices commute. \\

\subsection{The tangent category of dynamical systems}\label{secDynSystems}

To guarantee unique solutions to vector fields, we need initial conditions. Thus, it is useful to combine an ``initial condition'' (global element of the object) with a vector field; following standard terminology we call such a combination a dynamical system.  

\begin{definition} In a Cartesian tangent category $\X$: 
\begin{enumerate}[(i)]
\item A \textbf{dynamical system} is a triple $(M,V,g)$, where $M$ is an object, $V$ is a vector field on $M$, and $g$ is a global element of $M$:
	\[ 1 \to^{g} M \to^{V} T(M) \]
	The element $g$ is the \textbf{initial state} of the dynamical system and $V$ is the \textbf{differential transition}.
\item A \textbf{morphism} of dynamical systems $f: (M,V,g) \to (M',V',g')$ is a map $f: M \to M'$ which is a morphism of vector fields from $(M,V)$ to $(M',V')$, and sends $g$ to $g'$. That is, the following diagram commutes:
 \[ \xymatrix{1 \ar[r]^g \ar[dr]_{g'} & M \ar[d]_f \ar[r]^V & T(M) \ar[d]^{T(f)} \\ & M' \ar[r]_{V'} & T(M') } \]
 \end{enumerate}							
\end{definition}



\begin{proposition}
For any tangent category $\X$, its category of dynamical systems, ${\sf Dyn}\X$, is also a tangent category, with tangent functor
	\[ \bar{T}(M,F,g) := (TM,T(F)c,T(g)). \]
\end{proposition}
\begin{proof}
Follows as in Proposition \ref{propVFTanCat}.
\end{proof}

We may now ask what vector fields and dynamical systems look like in a context $\Gamma$ (Remark \ref{rmkContext}).  A vector field becomes a vector field in context; that is, a map $V: M \x \Gamma \to T(M)$ such that 
$V p = \pi_0$.  Then a dynamical system in context becomes an element in context (i.e. effectively a map from $\Gamma$)  and  a vector field in context:
\[  \Gamma \to^g   M  ~~~~~ M \x \Gamma \to^V T(M) \]
A morphism of dynamical systems $h: (M,V,g) \to (M',V',g')$ is now given by a map $h: M \x \Gamma \to M'$ such that the following diagrams commute: 
\[ \xymatrix{\Gamma \ar[r]^{\< g,1\>} \ar[rd]_{g'} & M \x \Gamma  \ar[d]^{h} \\ & M'} ~~~  
      \xymatrix{M \x \Gamma \ar[d]_{\<h,\pi_1\>} \ar[r]^{V \x 0} & T(M \x \Gamma) \ar[d]^{T(h)} \\ M' \x \Gamma \ar[r]_{V'} & T(M')} \]

Such dynamical systems will be useful in the next section.  

\section{Solutions to dynamical systems}\label{secSolutions}
 
 \subsection{Parameterized dynamical systems and their solutions}
 
Throughout the rest of this paper, we will fix a Cartesian tangent category $\X$, and a particular dynamical system on $C \in \X$:
	\[ 1 \to^{c_0} C \to^{c_1} T(C) \]
which we call a \textbf{curve object}: it will have various special properties which we shall develop below.  The key example to keep in mind in smooth manifolds is $C=\R$, $c_0$ the point $0$, and $c_1$ the vector field which assigns the multiplicative unit $1$ to each point in $\R$.   

Principal amongst the special properties we shall assume of $(C,c_1,c_0)$ (which is true of the above example) is that it should be a \textbf{preinitial} dynamical system {\em in all contexts\/}: this means explicitly that for any dynamical system $(F,g)$ in context $\Gamma$ there is at most one map $s$ such that
\[ \xymatrix{\Gamma \ar[r]^{\<!c_0,1\>} \ar[rd]_{k} & C \x \Gamma \ar@{..>}[d]^s \\ & M} ~~~  
   \xymatrix{C \x \Gamma \ar@{..>}[d]_{\< s, \pi_1\>} \ar[r]^{c_1 \x 0} & T(C \x \Gamma) \ar@{..>}[d]^{T(s)} \\ M \x \Gamma \ar[r]_H & T(M)} \]
Below we shall show how we can view the dynamical system as determining a differential equation and the map $s$ as being a solution to this dynamical system. Because of this perspective 
we shall sometimes refer to the requirement that the triangle commutes as the \textbf{initial condition} and the requirement that the square commutes as the \textbf{differential condition}.

However, to facilitate developing the properties of curve objects, it is useful to express its preinitial property in a slightly different form, by defining and speaking of ``parameterized dynamical systems''.

\begin{definition}
A \textbf{parameterized dynamical system} is a triple $(M,V,g)$, consisting of an object $M$, a vector field $V: M \to TM$, and a ``initial state'' $g: X \to M$. A \textbf{solution} to a parameterized dynamical system $(M,V,g)$ consists of a map $\gamma: C \times X \to M$ making the following diagram commute:
\[
\bfig
	\square<750,350>[C \times X`T(C \times X)`M`TM;c_1 \times 0`\gamma`T(\gamma)`V]
	\morphism(-500,350)<500,0>[X`C \times X;\<!c_0,1\>]
	\morphism(-500,350)|b|<500,-350>[X`M;g]
\efig
\]
\end{definition}
We shall often simply refer to parameterized dynamical systems as simply \emph{systems}.  

The following result relates solutions of parameterized dynamical systems to maps from $(C,c_0,c_1)$ in context.  

\begin{lemma} In a Cartesian tangent category $\X$: 
\begin{enumerate}[(i)]
	\item Let $(M,V,g: X \to M)$ be a parameterized dynamical system. Then $\gamma: C \times X \to M$ is a solution of $(M,V,g)$ if and only if $\gamma$ is a morphism in context $X$ from $(C,c_1,c_0)$ to $(M,\pi_0V, g)$.  
	\item Let 
	\[ X \to^{g} M, M \times X \to^{V} TM \]
	be a dynamical system in context $X$. Then a map $\alpha: C \times X \to M$ is a morphism from $(C,c_1,c_0)$ to $(M,g,V)$ in context $X$ if and only if $\<\alpha,\pi_1\>: C \times X \to M \times X$ is a solution to the parameterized dynamical system 
	\[ X \to^{\<g,1\>} M \times X \to^{\<V,\pi_10\>} T(M \times X). \]
\end{enumerate}
\end{lemma}
\begin{proof}
\begin{enumerate}[(i)]
	\item By definition, such a map $\gamma$ is a solution if and only if $\<!c_0,1\>\gamma = g$ and $(c_1 \times 0)T(\gamma) = \gamma V$. It is a morphism from $(C,c_1,c_0)$ to $(M,\pi_0V,g)$ in context $X$ if and only if $\<!c_0,1\>\gamma = g$ and $\<\gamma,\pi_1\>\pi_0V = (c_1 \times 0)T(\gamma)$. However, since $\<\gamma,\pi_1\>\pi_0V = \gamma V$, these two formulations are equivalent.
	\item By definition, such a map $\alpha$ is a morphism in context $X$ if and only if $\<!c_0,1\>\alpha = g$ and $(c_1 \times 0)T(\alpha) = \<\alpha,\pi_1\>V$. On the other hand, $\<\alpha,\pi_1\>$ is a solution to $(M \times X, \<V,\pi_10\>,\<g,1\>)$ if and only if the following diagram commutes:
\[
\bfig
	\square<750,350>[C \times X`T(C \times X)`M \times X`T(M \times X);c_1 \times 0`\<\alpha,\pi_1\>`T(\<\alpha,\pi_1\>)`\<V,\pi_10\>]
	\morphism(-600,350)<600,0>[X`C \times X;\<!c_0,1\>]
	\morphism(-600,350)|b|<600,-350>[X`M \times X;\<g,1\>]
\efig
\]	
However, note that post-composing the triangle by $\pi_1$ and the rectangle by $T(\pi_1)$, both equations are automatically satisfied; thus the above diagram commutes if and only the triangle and rectangle commute when post-composed by $\pi_0$ and $T(\pi_0)$. However, post-composing the triangle by $\pi_0$ gives the equation $\<!c_0,1\>\alpha = g$, and post-composing the square by $T(\pi_0)$ gives $(c_1 \times 0)T(\alpha) = \<\alpha,\pi_1\>V$; these are exactly the equations asking that $\alpha$ be a morphism in context $X$.  
\end{enumerate}
\end{proof}

\begin{corollary}\label{corInitiality}
The system $(C,c_0,c_1)$ has the property that it is preinitial in all contexts $X$ if and only if any solution to a parameterized dynamical system is unique.
\end{corollary}

\subsection{Examples in smooth manifolds}

To understand how a solution of a dynamical system (in the above abstract sense) is related to solving ordinary differential equations, we consider some basic examples in the category of smooth manifolds. As described in the previous section, the canonical curve object is $C = \mathbb{R}$ with $c_0 = 0$ and $c_1 = \<1_{\mathbb{R}}, 1\>$ (with the second $1$ here meaning the multiplicative unit in $\mathbb{R}$). For simplicity let us also consider a dynamical system on $M = \mathbb{R}$ with initial state $x_0: 1 \to \mathbb{R}$. A vector field on $\R$ is a map $V: \R \to T(\R) \cong \R \times \R$; since it is a vector field, this is entirely determined by its second (vector) component, and $V$ is thus determined by a smooth map $\hat{V}: \mathbb{R} \to \mathbb{R}$. A solution to this dynamical system is a map $s: \mathbb{R} \to \mathbb{R}$ such that the following diagram commutes:
\[
\bfig
	\square<600,350>[\R`T\R`\R`T\R;c_1`s`T(s)`V]
	\morphism(-600,350)<600,0>[1`\R;c_0]
	\morphism(-600,350)|b|<600,-350>[1`\R;x_0]
\efig
\]
 
 To make the triangle commute simply means that $s(0) = x_0$.  For the square to commute, since the codomain is $\R \cong \R \times \R$, it suffices to check the commutativity into each component. The commutativity of the first component is automatic for any $s$ since
	\[ c_1 T(s)p = c_1p s = s = s V p \]
(since $c_1$ and $V$ are vector fields). The commutativity of the second component is then telling us something about the derivative of $s$, namely that for any $x \in \R$, 
	\[ D(s)(x)\cdot 1 = s \hat{V} \]
that is, $s'(x) = \hat{V}(s(x))$. This precisely says that $s$ solves a first-order ordinary differential equation $\hat{V}$ with initial condition $s(0) = x_0$.

For example, if one takes the ``Euler vector field'' $V(x) = \<x,x\>$ (so that $\hat{V}(x) = x$) with initial condition $1$, then the above equation asks that $s'(x) = s(x)$ and $s(0) = 1$, which has the (unique) solution $s(x) = e^x$.

A time-dependent vector field on a manifold $M$ (for example, see \cite[pg. 451]{lee} is precisely a vector field in context $\mathbb{R}$, as by definition a time-dependent vector field $V$ on a manifold $M$ is a map
	\[ M \times \R \to^{V} TM \]
so that $Vp = \pi_0$.  

One can also get non-homogenous equations by working with time-dependent vector fields.  For example, to work with a non-homogenous differential equation like $y'(t) = y(t) + \cos(t)$, one can ask for a solution to the (time-dependant) vector field 
	\[ \<\pi_0c_1,\<\pi_1 + \pi_0\cos,1\>\>: \mathbb{R} \times \mathbb{R} \to T(\mathbb{R} \times \mathbb{R}). \]
A solution to this vector field would be a map $\gamma: \mathbb{R} \times \mathbb{R} \to \mathbb{R} \times \mathbb{R}$.  However, the first co-ordinate of such a solution could always be taken to be the identity map; the second co-ordinate would then give the desired solution to the original differential equation.   

Thus, the abstract setup provided in the previous section covers a variety of cases applicable to the solution of ordinary differential equations.

\subsection{Basic results on solutions}

For the rest of this section, we will prove some basic results about solutions to parameterized dynamical systems which we will often refer to as just {\em systems}.  

\begin{lemma}\label{lemmaZeroSolution}
For any map $g: X \to M$, the (parameterized dynamical) system $(M,0_M, g)$ is solved by
	\[ C \times X \to^{\pi_1} X \to^{g} M. \]
\end{lemma}
\begin{proof}
The initial condition equation follows since
	\[ \<!c_0,1\>\pi_1g = g \]
while the differential condition follows since
	\[ (c_1 \times 0)T(\pi_1g) = (c_1 \times 0)T(\pi_1)T(g) = \pi_1 0 T(g) = \pi_1 g 0 \]
by naturality of $0$.
\end{proof}
That is, the solution is ``for any time, stay in the same spot''.\\

\begin{lemma}\label{lemmaProductSolution}
If $(M_1,V_1,g_1)$ and $(M_2,V_2,g_2)$ are systems with respective solutions $\gamma_1, \gamma_2$, then the the system $(M_1 \times M_2, V_1 \times V_2, g_1 \times g_2)$ has solution
	\[ C \times M_1 \times M_2 \to^{\<\<\pi_0,\pi_1\>\gamma_1, \<\pi_0,\pi_2\>\gamma_2\>} M_1 \times M_2. \]
\end{lemma}
\begin{proof}
This is a straightforward exercise using the fact that $T$ preserves products.  
\end{proof}

Given a solution to a system, pre-composing by a new initial condition, i.e. changing context, still has a solution:  

\begin{lemma}\label{lemmaParamFromFlow}
Suppose $(M, V, g: X \to M)$ has a solution $\gamma: C \times M \to M$. Then for any $h: Y \to X$, the system $(M,V,hg)$ is solved by 
	\[ C \times Y \to^{(1 \times h)} C \times X \to^{\gamma} M. \]
\end{lemma}
\begin{proof}
We have the following diagram:
\[
\bfig
	\square<750,350>[Y`C \times Y`X`C \times X;\<!c_0,1\>`h`1 \times h`\<!c_0,1\>]
	\square(750,0)<750,350>[C \times Y`T(C \times Y)`C \times X`T(C \times X);c_1 \times 0``T(1 \times h)`c_1 \times 0]
	\square(750,-350)<750,350>[C \times X`T(C \times X)`M`TM;`\gamma`T(\gamma)`V]
	\morphism|b|<750,-350>[X`M;g]
\efig
\]
The top left square commutes by definition of pairing, the top right by naturality of $0$, and the bottom two diagrams commute by definition.  Thus $(1 \times h)\gamma$ solves $(M,V,hg)$.  
\end{proof}

Next, recall that if $V: M \to T(M)$ is a vector field, then 
	\[ T(M) \to^{T(V)} T^2(M) \to^{c} T^2(M) \]
is a vector field on $T(M)$, since $T(V)cp = T(V)T(p) = T(Vp) = T(1) = 1$ (this was also implicitly proven in Proposition \ref{propVFTanCat}).  The next result shows that solutions of systems with vector field $V$ give solutions of systems with vector field $T(V)c$.  This is known in some parts of the literature as the ``equation of variation for the flow'' (for example, see \cite[Section 1.4]{geoVectorFields}), since the result says that $T(V)c$ is solved by the partial derivative of the solution with respect to $M$.  An abstract way to express this result is to use the tangent functor on the category of dynamical systems:

%

\begin{lemma} \label{lemmaTFcSoln2} (Equation of variation) 
If $\gamma: C \times M \to M$ is a solution to the system $(M,V,g)$, then
	\[ C \times TX \to^{0 \x 1} T(C \times X) \to^{T(\gamma)} TM \]
is a solution to the tangent system $\bar{T}(M,V,g) = (TM,T(V)c,T(g))$.  
\end{lemma}

\begin{proof} 
Consider the following diagram:
\[
\bfig
	\square<800,350>[T(C \times X)`T^2(C \times X)`TM`T^2M;T(c_1 \times 0)`T(\gamma)`T^2(\gamma)`T(V)]
	\square(800,0)<800,350>[T^2(C \times X)`T^2(C \times X)`T^2M`T^2M;c``T^2(\gamma)`c]
	\square(0,350)/->`->`->`/<1600,350>[C \times TX`T(C \times TX)`T(C \times X)`T^2(C \times X);c_1 \times 0`0 \times 1`T(0 \times 1)`]
	\morphism(-800,700)<800,0>[TX`C \times TX;\<!c_0,1\>]
	\morphism(-800,700)|b|<800,-700>[TX`TM;T(g)]
\efig
\]
The triangle and bottom left square commute as they are $T$ applied to the solution diagram for $\gamma$.  The bottom right square commutes by naturality of $c$.  For the top rectangle, we have the following calculation:
\begin{eqnarray*}
&    & (0 \times 1)T(c_1 \times 0 )c \\
& = & (c_1 \times 1)(0 \times T(0))c \mbox{ (by naturality of $0$)} \\
& = & (c_1 \times 1)(0c \times T(0)c) \mbox{ (by naturality of $c$)} \\
& = & (c_1 \times 1)(T(0) \times 0) \mbox{ (by coherence of $c$)} \\
& = & (c_1 \times 0)T(0 \times 1)
\end{eqnarray*}
Thus the diagram commutes, and so $(0 \times 1)T(\gamma)$ is indeed a solution of $(TM,T(V)c,T(g))$.  
 \end{proof}

A further important property of a curve object is that the vector field $c_1$ should commute with itself, that is $c_1 T(c_1) = c_1T(c_1)c$ (see Definition \ref{defnCommuteVFs}).   In this case we have:

\begin{lemma}\label{lemmaTFcSoln1}
If $c_1$ commutes with itself and $\gamma: C \times M \to M$ solves the system $(M,V,g)$, then $\gamma V: C \times M \to TM$ solves the system $(TM,T(V)c,gV)$.     
\end{lemma}
\begin{proof}
The initial condition is straightforward.  For the differential condition, we need to show that
	\[ \gamma V T(V)c = (c_1 \times 0)T(\gamma)T(V). \]
Consider
\begin{eqnarray*}
&   & \gamma V T(V)c \\
& = & (c_1 \times 0)T(\gamma)T(V)c \\
& = & (c_1 \times 0)T(\gamma V) c \\
& = & (c_1 \times 0)T((c_1 \times 0)T(\gamma))c \\
& = & (c_1 \times 0)(T(c_1) \times T0)T^2(\gamma) c\\
& = & (c_1T(c_1)c  \times 0 T(0)c)T^2(\gamma) \mbox{ (naturality of $c$)} \\
& = & (c_1T(c_1)  \times 0 0)T^2(\gamma) \mbox{ ($c_1$ commutes with itself and coherence of $c$ )} \\
& = & (c_1T(c_1)  \times 0 T(0))T^2(\gamma) \mbox{ (naturality of $0$)} \\
& = & (c_1 \times 0)T((c_1 \times 0)T(\gamma)) \\
& = & (c_1 \times 0)T(\gamma)T(V)
\end{eqnarray*}
as required.  
\end{proof}

For differential objects, solutions of systems can be formulated in a way that makes their ``differential equation'' nature more readily apparent.  
\begin{proposition}\label{propDiffObjectSolns}
Suppose that $A$ is a differential object (with map $\hat{p}: TA \to A$) and has a dynamical system $(A,V,g)$ on it.   Then $\gamma: C \times A \to A$ solves $(A,V,g)$ if and only if the following diagram commutes:
\[
		\bfig
		\square/>`>`>`>/<750,350>[C \times X`T(C \times X)`A`A;c_1 \times 0`\gamma`D(\gamma)`\hat{V}]
		\morphism(-500,350)<500,0>[X`C \times X;\<!c_0,1\>]
		\morphism(-500,350)|b|<500,-350>[X`A;g]
		\efig
\]
(where $D(\gamma) := T(\gamma)\hat{p}$ and $\hat{V} := V\hat{p})$. 
\end{proposition}
\begin{proof}
Since $A$ is a differential object, $TA$ is a product with projections $\hat{p}$ and $p$.  Thus, the derivative condition holds if and only if it holds when post-composed by $\hat{p}$ and $p$.  However, for any $\gamma$ we have
	\[ (c_1 \times 0)T(\gamma)p = (c_1 \times 0)p\gamma = \gamma = \gamma V p \]
since $V$ is a vector field; that is, the derivative equation holds automatically when post-composed by $p$.  Moreover, when post-composed by $\hat{p}$, the derivative condition is precisely the rectangle in the diagram above.  
\end{proof}

\subsection{Higher-order systems and solutions}

In this section, we briefly describe how to define and work with higher-order dynamical systems.  

\begin{definition}
An \textbf{$n^{\rm th}$-order dynamical system on $M$} consists of a map $g: X \to T^{n-1}M$ (the initial condition) and an \textbf{$n^{\rm th}$-order vector field}, that is, a map $V: T^{n-1}M \to T^nM$ such that $Vp = VT(p) = \ldots = VT^{n-1}(p) = 1$.  
\end{definition}

For example, in $\mathbb{R}$ the $2^{\rm nd}$-order differential equation $\gamma'' + \gamma' + \gamma = 0$, $\gamma(0) = 0$, $\gamma'(0) = 1$ corresponds to the the $2^{\rm nd}$-order dynamical system $V(\<v,x\>) = \<v,x,-v,-x\>$, $g = \<0,1\>$.  

Note that a first-order dynamical system is just a dynamical system.  Moreover, note that if $(M,V,g)$ is an $n^{\rm th}$-order dynamical system, then $(TM,V,g)$ is an $(n-1)^{\rm th}$-order dynamical system, $(T^2M,V,g)$ is an $(n-2)^{\rm th}$-order dynamical system, etc.

\begin{definition}
For any map $\gamma: C \times X \to M$, inductively define $\gamma^{(n)}$ (the \textbf{$n^{\rm th}$-time derivative of $\gamma$}) by 
	\[ \gamma^{(0)} := \gamma \mbox{ and for each $n \geq 1$, } \gamma^{(n)} := (c_1 \times 0)T(\gamma^{(n-1)}): C \times X \to T^nM. \]
\end{definition}
Note that $\gamma^{(1)}$ is simply $(c_1 \times 0)T(\gamma)$, which is the expression appearing in the definition of a solution to a first-order dynamical system. We can thus make the following generalization of the notion of solution to $n^{\rm th}$-order systems:

\begin{definition}\label{defnNSolution}
A \textbf{solution} to an $n^{\rm th}$-order dynamical system $(M,V,g)$ consists of a map $\gamma: C \times X \to M$ such that 
\[
\bfig
	\node XC(0,0)[C \times X]
	\node X(-650,0)[X]
	\node TM1(0,-350)[T^{n-1}M]
	\node TM2(650,-350)[T^nM]
	\arrow[X`XC;\<!c_0,1\>]
	\arrow|b|[X`TM1;g]
	\arrow[XC`TM1;\gamma^{(n-1)}]
	\arrow|a|[XC`TM2;\gamma^{(n)}]
	\arrow|r|[TM1`TM2;V]
\efig
\]
\end{definition}

As is standard in differential geometry, we can reduce $n^{\rm th}$-order systems to lower-order ones:

\begin{proposition}
Suppose that $(M,V,g)$ is an $n^{\rm th}$-order dynamical system.  If $\gamma'$ is a solution to the $(n-1)^{\rm th}$-order dynamical system $(TM,V,g)$, then $\gamma := \gamma' p$ is a solution to the $n^{\rm th}$-order system $(M,V,g)$.
\end{proposition}
\begin{proof}
By definition of $\gamma'$, the following diagrams commute:
\[
\bfig
	\node XC(0,0)[C \times X]
	\node X(-650,0)[X]
	\node TM1(0,-350)[T^{n-2}M]
	\node TM2(650,-350)[T^{n-1}M]
	\arrow[X`XC;\<!c_0,1\>]
	\arrow|b|[X`TM1;g]
	\arrow[XC`TM1;\gamma'^{(n-2)}]
	\arrow|a|[XC`TM2;\gamma'^{(n-1)}]
	\arrow|r|[TM1`TM2;V]
\efig
\]
We want to show that $\gamma = \gamma' p$ satisfies the diagrams in Definition \ref{defnNSolution}. First, note that 
	\[ \gamma^{(1)} = (c_1 \times 0)T(\gamma' p) = (c_1 \times 0)T(\gamma')T(p) = \gamma'^{(1)} T(p) \]
so that more generally for any $1 \leq i \leq n$,
	\[ \gamma^{(i)} = \gamma'^{(i)} T^i(p). \]
Then for the left triangle in Definition \ref{defnNSolution}, consider
\begin{eqnarray*}
&  & \<!c_0,1\>\gamma^{(n-1)} \\
& = & \<!c_0,1\>\gamma'^{(n-1)} T^{n-1}(p) \\
& = & \<!c_0,1\>\gamma'^{(n-2)} V T^{n-1}(p) \mbox{ (definition of $\gamma'$)} \\
& = & \<!c_0,1\>\gamma'^{(n-2)} \mbox{ (by definition of $V$)} \\
& = & g \mbox{ (definition of $\gamma'$)} 
\end{eqnarray*}
as required.  

For the right triangle in Definition \ref{defnNSolution}, first consider
\begin{eqnarray*}
&  & \gamma^{(n-1)}V \\
& = & \gamma'^{(n-1)} T^{n-1}(p) V \\
& = & \gamma'^{(n-2)} V T^{n-1}(p) V \mbox{ (definition of $\gamma'$)} \\
& = & \gamma'^{(n-2)} V \mbox{ (by definition of $V$)} \\
& = & \gamma'^{(n-1)} \mbox{ (definition of $\gamma'$)}
\end{eqnarray*}
On the other hand
\begin{eqnarray*}
&  & \gamma^{(n)} \\
& = & \gamma'^{(n)} T^n(p) \\
& = & (c_1 \times 0)T(\gamma'^{(n-1)})T^n(p) \mbox{ (definition of the time derivative)} \\
& = & (c_1 \times 0)T(\gamma'^{(n-2)}V)T^n(p) \mbox{ (definition of $\gamma'$)} \\
& = & (c_1 \times 0)T(\gamma'^{(n-2)}VT^{n-1}(p)) \\
& = & (c_1 \times 0)T(\gamma'^{(n-2)}) \mbox{ (definition of $V$)} \\
& = & \gamma'^{(n-1)} \mbox{ (definition of the time derivative)}
\end{eqnarray*}
as required.  

\end{proof}

Thus we can always reduce the problem of solving $n^{\rm th}$-order dynamical systems to solving first-order ones:
\begin{corollary}\label{corHigherSolutions}
For any $n \geq 2$, if $(M,V,g)$ is an $n^{\rm th}$-order dynamical system and $\gamma'$ is a solution to the first order dynamical system $(T^{n-1},V,g)$, then $\gamma'pp \ldots p$ ($n-1$ applications of $p$) is a solution to the $n^{\rm th}$-order system $(M,V,g)$.
\end{corollary}
\begin{proof}
Simply apply the previous result $n-1$ times.
\end{proof}

\subsection{Geodesics for affine connections}\label{secGeodesics}

One of the most important properties of an affine connection (that is, a connection on the tangent bundle of a smooth manifold $M$) is that it generates geodesics: lines with zero acceleration (relative to the connection). This involves solving a second-order system. Here we briefly describe how to do this in a tangent category.   

Suppose that $(K,H)$ is a connection on the tangent bundle of $M$ (see \cite[Definition 5.2]{connections}).  Then we have a map 
	\[ TM \to^{S := \<1,1\>H} T^2M \]
which is a second-order vector field since
	\[ \<1,1\>HT(p) = \<1,1\>\pi_0 = 1 = \<1,1\>\pi_1 = \<1,1\>Hp. \]

If $\gamma: C \times M \to M$ is a solution to the second order dynamical system $(M,H,1_{TM})$, then we call $\gamma$ the \textbf{geodesic flow} associated to the connection. By definition, this map has the property that 
\[
\bfig
	\node VC(0,0)[C \times M]
	\node TM1(0,-350)[TM]
	\node TM2(650,-350)[T^2M]
	\arrow[VC`TM1;\gamma']
	\arrow|a|[VC`TM2;\gamma'']
	\arrow|r|[TM1`TM2;\<1,1\>H]
\efig
\]
commutes. In particular, the second derivative of $g$ with respect to the connection, i.e., 
	\[ g''K: C \times M \to TM \]
has the property that it is constantly zero, since
	\[ g''K = g'\<1,1\>HK = g'\<1,1\>\pi_1p0 = g'p0 = g0. \]
That is, the geodesic flow has zero acceleration relative to the connection, which is the defining property of geodesics (for example, see \cite[pg. 246]{spivak2}).  


\section{Curve objects and flows}\label{secCurveAndFlows}

\subsection{Curve objects}

We are now ready to define the central notion of this paper: a curve object.  

\begin{definition}\label{defnCurveObject}
In a Cartesian tangent category $\X$ a \textbf{curve object} is a dynamical system $(C,c_1,c_0)$ so that:
\begin{itemize}
	\item \textbf{(Preinitial)} It is a preinitial dynamical system in all contexts; as in Corollary \ref{corInitiality}, this is equivalent to asking that any solution $\gamma$ to a parameterized dynamical system $(M, V, g)$ be unique:
	\[
		\bfig
		\square/>`.>`.>`>/<750,350>[C \times X`T(C \times X)`M`TM;c_1 \times 0`\gamma`T(\gamma)`V]
		\morphism(-500,350)<500,0>[X`C \times X;\<!c_0,1\>]
		\morphism(-500,350)|b|<500,-350>[X`M;g]
		\efig
	\]
	\item \textbf{(Commutativity of $c_1$)} $c_1T(c_1)c = c_1T(c_1)$.
	\item \textbf{(Completeness of $c_1$)} The system $(C,c_1,1_C)$ has a (unique) solution, which we write as $\pc: C \times C \to C$, thus we have the following commutative diagram:  
	\[
	\bfig
		\square/>`.>`.>`>/<750,350>[C \times C`T(C \times C)`C`TC;c_1 \times 0`\pc`T(\pc)`c_1]
		\morphism(-500,350)<500,0>[C`C \times C;\<!c_0,1\>]
		\morphism(-500,350)|b|<500,-350>[C`C;1_C]
	\efig
	\]
\end{itemize}
\end{definition}

\begin{example} 
In the tangent category of smooth manifolds ($\sman$) $\mathbb{R}$ is a curve object, with structure described as in the previous section: $c_0$ is the point 0, while $c_1$ is the vector field whose vector value is constantly the multiplicative unit in $\mathbb{R}$. (This vector field is sometimes written as $\frac{\partial}{\partial x}$). The initial condition is the standard uniqueness of solutions result for ODEs: for example, see Theorem 17.9 in \cite{lee}. The solution of $(C,c_1,1_c)$ is $(t,x) \mapsto t+x$. It is easy to directly check that $c_1$ satisfies $c_1T(c_1)c = c_1T(c_1)$, but it also follows from Corollary \ref{corFlipCondition}. It is also well-known that this curve object does not have solutions for all dynamical systems: for example, the vector field $x \mapsto (x,x^2)$ on $\mathbb{R}$ has no solution that exists for all time.  
\end{example}

\begin{example}
If $\smooth$ is the full subcategory of $\sman$ consisting of just the finite powers of $\mathbb{R}$, and $\poly$ is the subcategory of $\smooth$ consisting of just the polynomial functions between finite powers of $\mathbb{R}$, then $\mathbb{R}$, with structure as in the previous example, is also a curve object in these tangent categories (note that the solution to $(C,c_1,1_C)$ is polynomial).  
\end{example}

\begin{example}
More generally, $\mathbb{R}$ is a curve object in the tangent category of smooth manifolds modelled on Banach spaces (see \cite[Section IV.1]{langDiffGeometry}).
\end{example}

\begin{remark}
In the tangent category of convenient vector spaces (or, more generally, manifolds modelled on convenient vector spaces), $\mathbb{R}$ is \textbf{not} a curve object. Quoting \cite[pg. 329]{convenient}, ``...the classical results on existence and uniqueness of solutions of equations like the inverse function theorem, the implicit function theorem, and the Picard-Lindelof theorem on ordinary differential equations can be deduced essentially from one another, and all depend on Banach's fixed point theorem. Beyond Banach spaces these proofs do not work anymore...''. In particular, \cite[Example 2, pg. 330]{convenient}, gives an example of non-uniqueness of solutions to a dynamical system.  
\end{remark}

\begin{example}
$\mathbb{C}$ is a curve object in the tangent category of analytic maps between $\mathbb{C}^n$'s: see \cite[Theorem 2.5.1]{hille} for uniqueness.
\end{example}

\begin{example} 
If $\mathbb{E}$ is a model of SDG, then the object $D_{\infty}$ is a \textbf{total} curve object - in the sense that it has solutions for all dynamical systems - for the tangent category of microlinear objects in $\mathbb{E}$: see \cite[Theorem 2.4]{kockReyesSolutionsODEs}. 
\end{example}

\begin{example}
We shall see later (Proposition \ref{propVFCurveObject}) that if $(C,c_1,c_0)$ is a curve object in $\X$, then $((C,0_C),c_1,c_0)$ is a curve object in $\VF$. Related to this, we shall also see that a category of flows in $\X$ has a curve object: see Corollary \ref{corFLOWCurveObject}.  
\end{example}

For much of the rest of this paper, we fix a curve object $(C,c_1,c_0)$ in a Cartesian tangent category $\X$; when we speak of solutions of a dynamical system, we mean in reference to this fixed curve object. \\

Note that we are not assuming that the curve object has solutions for all dynamical systems: as noted above, there are vector fields even on $\R$ which do not have solutions for all time. Following standard differential geometry terminology, those vector fields which do have solutions for all time we call complete.

\begin{definition}
A vector field, $V: M \to TM$, is said to be \textbf{complete} in case every parameterized dynamical system of the form $(M,V,g)$ has a solution. We shall write $\CVF$ for the full subcategory of $\VF$ consisting of the complete vector fields.  
\end{definition}

\begin{corollary}\label{corAllSolutions}
A vector field $V: M \to T(M)$ is complete if and only if the parameterized dynamical system $(M,V,1_M)$ has a (unique) solution $\gamma$:
\[
		\bfig
		\square/>`.>`.>`>/<750,350>[C \times M`T(C \times M)`M`TM;c_1 \times 0`\gamma`T(\gamma)`V]
		\morphism(-500,350)<500,0>[M`C \times M;\<!c_0,1\>]
		\morphism(-500,350)|b|<500,-350>[M`M;1_M]
		\efig
\]
   
\end{corollary}

\begin{proof}
A complete vector field necessarily has a solution for $(M,V,1_M)$ so the only difficulty is to prove the converse. Suppose, therefore, that $\gamma$ is a solution to $(M,V,1_M)$, and that $g: X \to M$.  Then, by Lemma \ref{lemmaParamFromFlow}, $(M,V,g)$ is solved by $(1 \times g)\gamma$, so that $V$ is complete.  
\end{proof}

For this reason, when we speak of a solution to a vector field $V: M \to TM$, we mean a solution to the system $(M,V,1_M)$.  

\begin{example}
By Lemma \ref{lemmaZeroSolution}, for any object $M$, the vector field $0_M: M \to TM$ is complete, with solution $\pi_1: C \times M \to M$.  
\end{example}

\begin{example}
By Lemma \ref{lemmaTFcSoln2}, if $V: M \to TM$ is complete with solution $\gamma: C \times M \to M$, then its tangent vector field, $T(V)c: TM \to T^2M$ is also complete, with solution 
	\[ C \times TM \to^{0 \times 1} T(C \times M) \to^{T(\gamma)} TM. \]
\end{example}

\begin{proposition}\label{propCVFTanCat}
The category $\CVF$ is a Cartesian tangent category, with tangent structure as in $\VF$. 
\end{proposition}
\begin{proof}
By definition, $\CVF$ is a full subcategory of $\VF$, so it suffices to show that the tangent and product structure on $\VF$ restricts to the complete vector fields. However, the previous example shows if a vector is complete then so its associated tangent vector field, and Lemma \ref{lemmaProductSolution} shows that if $V_1$ and $V_2$ are complete vector fields then so is their product.
\end{proof}

The following result says that relative to a curve object, ``every (smooth) vector field is invariant under its own [solution]'' (\cite{lee}, pg. 442). It is our first use of the preinitial assumption for a curve object.  

\begin{proposition}\label{lemmaOtherDerivativeFlow} 
In a tangent category $\X$ with curve object $C$, if $V$ is a vector field on $M$ and $\gamma$ is a solution of $V$, then the following diagram commutes:
\[
\bfig
	\square<700,350>[C \times M`T(C \times M)`M`TM;0 \times V`\gamma`T(\gamma)`V]
\efig
\]
\end{proposition}
\begin{proof}
Since $c_1$ commutes with itself, by Lemma \ref{lemmaTFcSoln1} (with $g=1_M$), $\gamma V$ is a solution to the system $(TM,T(V)c,V)$. But by Lemma \ref{lemmaTFcSoln2} and Lemma \ref{lemmaParamFromFlow}, $(0 \times V)T\gamma$ is also a solution to $(TM,T(V)c,V)$. Thus, by uniqueness of solutions, $(0 \times V)T(\gamma) = \gamma V$, as required.  
\end{proof}

\subsection{Additive structure of a curve object}

We have not assumed any additive structure on the curve object $C$. However, being able to ``add values of time'' is frequently used in statements about solutions of differential equations. In particular, it is a standard result that if $\gamma$ is a solution of a vector field, then $\gamma(t_1,\gamma(t_2,x)) = \gamma(t_1 + t_2,x)$.  In this section, we will show that the curve object assumptions are enough to guarantee the existence of additive structure on $C$, and equations like the above one hold. In almost all of the following proofs, the uniqueness of solutions assumption for a curve object will play a prominent role.  

By the axioms for a curve object, we have a (unique) solution $\pc: C \times C \to C$ to the vector field $c_1$. By definition, then, the following diagram commutes:
\[
\bfig
	\square<750,350>[C \times C`T(C \times C)`C`TC;c_1 \times 0`\pc`T(\pc)`c_1]
	\morphism(-500,350)<500,0>[C`C \times C;\<!c_0,1\>]
	\morphism(-500,350)|b|<500,-350>[C`C;1_C]
\efig
\]

Our goal in this section is to show that $(C,\pc,c_0)$ is a commutative monoid.

\begin{lemma}\label{lemmaPlusProps}
With $\pc$ defined as above, we have
\begin{enumerate}[(i)]
	\item $\<!c_0,1_C\>\pc = 1_C$;
	\item $\pc c_1 = (c_1 \times 0)T(\pc)$;
	\item $\pc c_1 = (0 \times c_1)T(\pc)$.
\end{enumerate}
\end{lemma}
\begin{proof}
The first two equations are exactly the diagrams above, while the third equation is by Proposition \ref{lemmaOtherDerivativeFlow}.
\end{proof}

Note that (i) only tells us that $c_0$ is a $\emph{left}$ unit for $\pc$. We next prove that it is a right unit.  

\begin{lemma}\label{lemmaPlusRightUnit}
$c_0$ is a right unit for $\pc$; that is $ \<1,!c_0\>\pc = 1$.
\end{lemma}
\begin{proof}
We will show that $\<1,!c_0\>\pc$ is a solution to the system $(C,c_1,c_0)$. Consider the following diagram:
\[
\bfig
	\square<750,350>[C`TC`C \times C`T(C \times C);c_1`\<1,!c_0\>`T(\<1,!c_0\>)`c_1 \times 0]
	\square(0,-350)<750,350>[C \times C`T(C \times C)`C`TC;`\pc`T(\pc)`c_1]
	\morphism(-600,350)<600,0>[1`C;c_0]
	\morphism(-600,350)|b|<600,-700>[1`C;c_0]
\efig
\]

The triangle commutes by using Lemma \ref{lemmaPlusProps}(i), the bottom square by Lemma \ref{lemmaPlusProps}(ii), and the top square since naturality of $0$ gives $T(!c_0) = !c_00$.  

Thus $\<1,!c_0\>\pc$ is a solution to $(c_1,c_0)$. But $1_c$ is also a solution to this system, and so by uniqueness of solutions, $\<1,!c_0\>\pc = 1$.  
\end{proof}

The next result shows that $\pc$ is commutative; it makes use of the right unit result.  

\begin{lemma}\label{lemmaPlusCommutative}
$\pc$ is a commutative operation; that is $\<\pi_1,\pi_0\>\pc = \pc$.
\end{lemma}
\begin{proof}
Consider the following diagram:
%
\[
\bfig
	\square<750,350>[C \times C`T(C \times C)`C \times C`T(C \times C);c_1 \times 0`\<\pi_1,\pi_0\>`T(\<\pi_1,\pi_0\>)`0 \times c_1]
	\square(0,-350)<750,350>[C \times C`T(C \times C)`C`TC;`\pc`T(\pc)`c_1]
	\morphism(-600,350)<600,0>[C`C \times C;\<!c_0,1\>]
	\morphism(-600,350)|b|<600,-700>[C`C;1_C]
\efig
\]
The triangle commutes by Lemma \ref{lemmaPlusRightUnit}, the top square since $T$ preserves products, and the bottom square by Lemma \ref{lemmaPlusProps}(iii) (which uses Lemma \ref{lemmaOtherDerivativeFlow}).  

Thus $\<\pi_1,\pi_0\>\pc$ solves the same system as $\pc$. Thus by uniqueness of solutions, $\<\pi_1,\pi_0\>\pc = \pc$.

\end{proof}

We now turn to the associativity of $\pc$, first proving the analog of the equation mentioned earlier (for a solution $\gamma$, $\gamma(t_1,\gamma(t_2,x)) = \gamma(t_1 + t_2,x)$).  

\begin{proposition}\label{assocResults}
~
\begin{enumerate}[(i)]
	\item If $\gamma$ is a solution of a vector field $V$ on $M$, then 
		$ (\pc \times 1)\gamma = (1 \times \gamma)\gamma$. 
	\item $\pc$ is associative; that is
		$(\pc \times 1)\pc = (1 \times \pc)\pc$. 
	\end{enumerate}
\end{proposition}
\begin{proof}~
\begin{enumerate}[{\em (i)}]
\item Consider the diagram
\[
\bfig
	\square<900,350>[C \times C \times M`T(C \times C \times M)`C \times M`T(C \times M);c_1 \times 0`\pc \times 1`T(\pc \times 1)`c_1 \times 0]
	\square(0,-350)<900,350>[C \times M`T(C \times M)`M`TM;`\gamma`T(\gamma)`V]
	\morphism(-750,350)<750,0>[C \times M`C \times C \times M;\<!c_0,1\>]
	\morphism(-750,350)|b|<750,-700>[C \times M`M;\gamma]
\efig
\]
The triangles commutes by Lemma \ref{lemmaPlusProps}(i), the top square by Lemma \ref{lemmaPlusProps}(ii), and the bottom square by definition of $\gamma$. Thus $(\pc \times 1)\gamma$ solves the system $(M, V,\gamma)$. However, by Lemma \ref{lemmaParamFromFlow}, $(1 \times \gamma)\gamma$ also solves $(M, V,\gamma)$. Thus, by uniqueness of solutions, $(\pc \times 1)\gamma = (1 \times \gamma)\gamma$, as required.  
\item Apply {\em (i)} to the vector field $c_1$ and the solution $\pc$.
\end{enumerate}
\end{proof}

\begin{theorem}\label{thmCurveMonoidStructure}
$(C,\pc,c_0)$ is a commutative monoid.
\end{theorem}
\begin{proof}
This follows from Lemma \ref{lemmaPlusProps}, Lemma \ref{lemmaPlusRightUnit}, Lemma \ref{lemmaPlusCommutative}, and Proposition \ref{assocResults}(ii).  
\end{proof}

\subsection{Flows and complete vector fields}

With additive structure on $C$, we can define a flow on an object $M$.  

\begin{definition}
Let $M$ be an object.  Say that a map $\gamma: C \times M \to M$ is a \textbf{flow on $M$} if $\gamma$ is an action of $(C,\pc,c_0)$ on $M$; that is, if the following unit and associativity diagrams commute:
\[
\bfig
	\qtriangle<500,350>[M`C \times M`M;\<!c_0,1\>`1_M`\gamma]
	\square(1250,0)<750,350>[C \times C \times M`C \times M`C \times M`M;1 \times \gamma`\pc \times 1`\gamma`\gamma]
\efig
\]

\end{definition}

In this section, we establish that flows on $M$ are in bijective correspondence with complete vector fields.  

\begin{proposition}\label{definitionGamma}
If $F$ is a complete vector field on $M$, its (unique) solution, which we write as $\gamma(F): C \times M \to M$, is a flow on $M$.  
\end{proposition}
\begin{proof}
By definition $\gamma(F)$ satisfies the unit condition, and by Proposition \ref{assocResults}(i), $\gamma(F)$ satisfies the associativity condition.  
\end{proof}

To go from flows to complete vector fields, we ``differentiate the flow at time 0'':

\begin{definition}\label{definitionIota}
Given a flow $\gamma: C \times M \to M$, define $\iota(\gamma)$ to be the derivative of the flow at time 0; that is:
	\[ M \to^{\iota(\gamma)} TM := M \to^{\<!c_0,1\>} C \times M \to^{(c_1 \times 0)} T(C \times M) \to^{T(\gamma)} TM. \]
\end{definition}
This vector field is also known as the infinitesimal generator of the flow (see \cite[pg. 439]{lee}).  

\begin{proposition}
For any flow $\gamma: C \times M \to M$, $\iota(\gamma)$ is a complete vector field with solution $\gamma$.
\end{proposition}
\begin{proof}
That $\iota(\gamma)$ is a vector field is straightforward using naturality of $p$ and the unit condition for a flow:
	\[ \<!c_0,1\>(c_1 \times 0)T(\gamma)p = \<!c_0,1\>(c_1 \times 0)p \gamma =\<!c_0,1\>\gamma = 1_M. \]
To show that $\gamma$ is a solution to $\iota(\gamma)$, we need to show that $\gamma \iota(\gamma) = (c_1 \times 0)T(\gamma)$ (the initial condition is automatic).  For this, first note that by by applying $T$ to the associativity requirement of a flow, we get 
	\[ (1 \times T(\gamma))T(\gamma) = (T(\pc) \times 1)T(\gamma) \ (\star). \]
Now consider
	\begin{eqnarray*}
		\gamma \iota(\gamma) & = & \gamma \<!c_0,1\>(c_1 \times 0)T(\gamma) \\
		& = & \<!c_0c_1, \gamma 0\>T(\gamma) \\
		& = & \<!c_0c_1, (0 \times 0)T(\gamma)\>T(\gamma) \mbox{ (naturality of $0$)} \\
		& = & \<!c_0c_1, (0 \times 0)\>(1 \times T(\gamma))T(\gamma) \\
		& = & \<!c_0c_1, (0 \times 0)\>(T(\pc) \times 1)T(\gamma) \mbox{ (by $\star$)} \\
		& = & (\<!c_0c_1,0\>T(\pc) \times 0)T(\gamma) \\
		& = & (\<!c_0,1\>(c_1 \times 0)T(\pc) \times 0)T(\gamma) \\
		& = & (\<!c_0,1\>\pc c_1 \times 0)T(\gamma) \mbox{ (since $\pc$ is a solution of $c_1$)} \\
		& = & (c_1 \times 0)T(\gamma) \mbox{ (since $\pc$ is a solution of $c_1$)} 
	\end{eqnarray*}
Thus $\gamma$ is indeed a solution of $\iota(\gamma)$.  
\end{proof}

\begin{theorem}\label{thmCVFFlow}
For any object $M$, the functions $\gamma$ (\ref{definitionGamma}) and $\iota$ (\ref{definitionIota}) establish a bijection between the set of complete vector fields on $M$ and the set of flows on $M$.
\end{theorem}
\begin{proof}
The previous proposition established not only that $\iota$ was well-defined, but also that $\gamma$ is itself a solution to $\iota(\gamma)$, so that by uniqueness of solutions, $\gamma(\iota(\gamma)) = \gamma$. For the other direction, consider 
\begin{eqnarray*}
&  & \iota(\gamma(F)) \\
& = & \<!c_0,1\>(c_1 \times 0)T(\gamma(F)) \\
& = & \<!c_0,1\>\gamma(F) F \mbox{ (since $\gamma(F)$ is a solution of $F$)} \\
& = & F \mbox{ (since $\gamma(F)$ is a solution of $F$)}
\end{eqnarray*}
Thus $\iota(\gamma(F)) = F$, and so $\gamma$ and $\iota$ are inverses.
\end{proof}

\subsection{The tangent category of flows}

In this section, we describe how flows form a tangent category.  As we shall see, just as vector fields in the tangent category of vector fields were interesting (Proposition \ref{propCommuteVFs}) so too are vector fields in the tangent category of flows.    

\begin{definition}
If $\gamma_1$ is a flow on $M_1$ and $\gamma_2$ is a flow on $M_2$, then a \textbf{flow morphism} from $(M_1,\gamma_1)$ to $(M_2, \gamma_2$) is an action-preserving map; that is, a map $f: M_1 \to M_2$ so that the following diagram commutes:
\[
\bfig
	\square<500,350>[C \times M_1`M_1`C \times M_2`M_2;\gamma_1`1 \times f`f`\gamma_2]
\efig
\]
Let $\FLOW$ denote the category whose objects are pairs $(M,\gamma)$ and whose arrows are flow morphisms.  
\end{definition}

Just as flows correspond to complete vector fields, so too do maps between them.  
\begin{proposition}\label{propVFMaps}
Suppose that $M_1$ and $M_2$ are objects with complete vector fields $V_1,V_2$, which have corresponding flows $\gamma_1$, $\gamma_2$.  Then $f$ is a map from $(M_1,V_1)$ to $(M_2,V_2)$ in $\CVF$ if and only if $f$ is a map from $(M_1,\gamma_1)$ to $(M_2, \gamma_2)$ in $\FLOW$.  That is,  
\[
\bfig
	\square<500,350>[M_1`TM_1`M_2`TM_2;V_1`f`T(f)`V_2]
	\place(250,175)[=]
	\place(850,175)[\Leftrightarrow]
	\square(1200,0)<500,350>[C \times M_1`M_1`C \times M_2`M_2;\gamma_1`1 \times f`f`\gamma_2]
	\place(1450,175)[=]
\efig
\]
\end{proposition}
\begin{proof}
First, suppose that $f$ preserves the vector fields.  By Lemma \ref{lemmaParamFromFlow}, $(1 \times f)\gamma_2$ solves the system $(V_2,f)$.  Thus, it suffices to show that $\gamma_1f$ also solves this system.  By definition of $\gamma_1$, $\<!c_0,1\>\gamma_1f = 1_{M_1} f = f$, so the initial condition is satisfied.   For the derivative condition, consider
	\[ (c_1 \times 0)T(\gamma_1 f) = (c_1 \times 0)T(\gamma_1) T(f) = \gamma_1 V_1 T(f) = \gamma_1 f V_2 \]
since $f$ is a vector field morphism.  Thus $(1 \times f)\gamma_2 = \gamma_1 f$, so $f$ is a flow morphism.  

Conversely, suppose that $f$ is a flow morphism, so $(1 \times f)\gamma_2 = \gamma_1 f$.  Taking the time derivative at 0 of both sides of this equation gives
	\[ \<!c_0,1\>(c_1 \times 0)T(1 \times f)T(\gamma_2) = \<!c_0,1\>(c_1 \times 0)T(\gamma_1)T(f). \]
However, using naturality of $0$, the left side of this equation is equal to $f\<!c_0,1\>(c_1 \times 0)T(\gamma_2)$, which equals $f V_2$ by definition of $\gamma_2$.  Moreover, by definition of $\gamma_1$, the right side immediately equals $V_1T(f)$.  Thus $f V_2 = V_1 T(f)$, so that $f$ is a flow morphism, as required. 
\end{proof}

Combining Theorem \ref{thmCVFFlow} and Proposition \ref{propVFMaps}, we get:

\begin{corollary}\label{corIsoCVF_Flow}
\begin{enumerate}[(i)]
	\item The categories $\CVF$ and $\FLOW$ are isomorphic via the functors
	\[ (M,V) \mapsto (M, \gamma(V)), f \mapsto f, \]
and
	\[ (M, \gamma) \mapsto (M,\iota(\gamma)), f \mapsto f. \]
	\item $\FLOW$ is a Cartesian tangent category, with tangent functor
	\[ \bar{T}(M,\gamma) = (TM,(0 \times 1)T(\gamma)) \]
and product
	\[ (M_1,\gamma_1) \times (M_2,\gamma_2) = (M_1 \times M_2, \<\<\pi_0,\pi_1\>\gamma_1, \<\pi_0,\pi_2\>\gamma_2\>). \]
	\item $\FLOW$ and $\VF$ are isomorphic as Cartesian tangent categories.
\end{enumerate}
\end{corollary}
\begin{proof}
(i) follows by the results above. For (ii), by Proposition \ref{propCVFTanCat}, complete vector fields form a tangent category, with tangent functor $\bar{T}(M,V) = (TM, T(V)c)$.  Moreover, by the ``equation of variation'' (Lemma \ref{lemmaTFcSoln2}) if $\gamma$ is the solution to $V$, then $(0 \times 1)T(\gamma)$ is the solution to $T(V)c$, giving the tangent structure on $\FLOW$.  The product structure follows by Lemma \ref{lemmaProductSolution}.  (iii) then follows by definition.  
\end{proof}

The following is a standard definition in differential geometry (eg., see \cite[pg. 442]{lee}): 

\begin{definition}
Suppose that $(M,V)$ is a vector field and $(M,\gamma)$ a flow. Say that \textbf{$V$ is invariant under $\gamma$} if the following diagram commutes:
\[
\bfig
	\square<750,350>[C \times M`T(C \times M)`M`TM;0 \times V`\gamma`T(\gamma)`V]
\efig
\]
\end{definition}

However, this definition naturally arises by considering vector fields in the tangent category $\FLOW$:

\begin{proposition}\label{propVFInFLOW}
A vector field $((M,\gamma),V)$ in the tangent category $\FLOW$ is precisely a vector field $(M,V)$ in $\X$ which is invariant under $(M,\gamma)$.
\end{proposition}
\begin{proof}
Suppose we are given a vector field $V: (M,\gamma) \to \bar{T}(M,\gamma) = (TM,(0 \times 1)T(\gamma))$ in $\FLOW$. Then $V$ is a vector field in $\X$, but also, since it is a map in $\FLOW$, it must make the following diagram commute:
\[
\bfig
	\square<500,350>[C \times M`M`C \times TM`TM;\gamma`1 \times V`V`(0 \times 1)T(\gamma)]
\efig
\]
However, $(1 \times V)(0 \times 1)T(\gamma) = (0 \times V)T(\gamma)$, so commutativity of this diagram is the same as the commutativity required for $V$ to be invariant under $\gamma$. Thus $V$ is a vector field invariant under $\gamma$; conversely, a vector field which is invariant under $\gamma$ gives a vector field on $(M,\gamma)$ in $\FLOW$.
\end{proof}

This immediately gives us part of the central theorem on commutation of flows:

\begin{corollary}\label{corCommuteIsoInvariant}
For complete vector fields $(M,V_1)$, $(M,V_2)$, the following are equivalent:
\begin{enumerate}[(i)]
	\item $V_1$ and $V_2$ commute;
	\item $V_1$ is invariant under the flow of $V_2$;
	\item $V_2$ is invariant under the flow of $V_1$.
\end{enumerate}
\end{corollary}
\begin{proof}
By Proposition \ref{propCommuteVFs}, a pair of commuting complete vector fields $(V_1,V_2)$ is equivalent to $V_2$ being a vector field on $(M,V_1)$ in $\CVF$. However, by Corollary \ref{corIsoCVF_Flow}, $\CVF$ is isomorphic as a tangent category to $\FLOW$, so a vector field $V_2$ on $(M,V_1)$ is equivalent to $V_2$ being a vector field on $(M,\gamma(V_1))$ in $\FLOW$. By Proposition \ref{propVFInFLOW}, this is equivalent to $V_2$ being invariant under the flow of $V_1$. Thus (i) and (iii) are equivalent. Moreover, by Lemma \ref{lemmaCommSymmetric}, commutation of vector fields is a symmetric relation, and so these are also equivalent to (ii).
\end{proof}

Combining this with an earlier result, we can show that commutation of complete vector fields is reflexive in any tangent category with a curve object:
\begin{corollary}
In a tangent category with a curve object, any complete vector field $V: M \to TM$ commutes with itself.
\end{corollary}
\begin{proof}
By Proposition \ref{lemmaOtherDerivativeFlow}, every complete vector field is invariant under its own flow; thus, by Corollary \ref{corCommuteIsoInvariant}, every complete vector field commutes with itself.
\end{proof}

\subsection{Commuting flows and vector fields}\label{secCommFlowsAndVFs}

In this section, we explore how to detect when two flows commute in a tangent category.  

\begin{definition}
Two flows $(M,\gamma_1)$, $(M,\gamma_2)$ are said to \textbf{commute} if their order of application does not matter; that is, if the following diagram commutes:
\[
\bfig
	\square/`>`>`>/<1700,350>[C \times C \times M`C \times M`C \times M`M;`1 \times \gamma_2`\gamma_2`\gamma_1]
	\morphism(0,350)<850,0>[C \times C \times M`C \times C \times M;\tau \times 1]
	\morphism(850,350)<850,0>[C \times C \times M`C \times M;1 \times \gamma_1]
\efig
\]  
(where $\tau := \<\pi_1,\pi_0\>$, the ``twist'' map).  
\end{definition}

Our main goal in this section is to prove the following theorem, which is a standard result in ordinary differential geometry (for example, see \cite[Prop. 18.5]{lee}):
\begin{theorem}\label{thmCommFlows}
Suppose that $V_1$ and $V_2$ are vector fields with associated flows $\gamma_1, \gamma_2$. Then the following are equivalent:
\begin{enumerate}[(i)]
	\item The flows $\gamma_1, \gamma_2$ commute.
	\item $V_1$ is invariant under the flow of $V_2$.
	\item $V_2$ is invariant under the flow of $V_1$.
	\item $V_1$ and $V_2$ commute (as vector fields).
\end{enumerate}
\end{theorem}

We have seen that the equivalence of (ii), (iii) and (iv) follows immediately from considering vector fields in the tangent categories of vector fields and flows (Corollary \ref{corCommuteIsoInvariant}). We shall next see that a pair of commuting flows is equivalent to a flow in the tangent category of flows. However, for this to even make sense, we first have to see how the tangent category of vector fields (and hence also of flows) itself has a curve object.  

\begin{proposition}\label{propVFCurveObject}
If $(C,c_1,c_0)$ is a curve object in $\X$, then $((C,0),c_1,c_0)$ is a curve object in both $\VF$ and $\CVF$. Moreover, for a dynamical system $((M,V_1)),V_2,g)$ in $\VF)$, $\gamma: (C \times M, 0 \times V_1) \to (M,V_1)$ is a solution to $((M,V_1)),V_2,g)$ in $\VF$ if and only if $\gamma$ is a solution to $(M,V_2,g)$ in $\X$.
\end{proposition}
\begin{proof}
First, since $0$ is a vector field, $(C,0)$ is an object in $\VF$, but it is also an object in $\CVF$ by Lemma \ref{lemmaZeroSolution}. Also, note that $c_1$ is a map from $(C,0)$ to $\bar{T}(C,0)$ in $\VF$ and $c_0$ is a map from $(1,0)$ to $(C,0)$ by naturality of $0$ (in both cases).  

Next, we will prove the ``moreover'' statement. The forward direction is immediate, by applying the forgetful functor $U: \VF \to \X$ to the solution. Conversely, suppose that $\gamma$ is a solution to $(M,V_2,g)$ in $\X$. We then need $\gamma$ to be a map from $(C \times M, 0 \times V_1)$ to $(M,V_1)$ in $\VF$. However, this asks that the following diagram commute:
\[
\bfig
	\square<500,350>[C \times M`M`T(C \times M)`TM;\gamma`0 \times V_1`V_1`T(\gamma)]
\efig
\]
which is precisely the requirement that $V_1$ be invariant under $\gamma$. But this follows since $V_2$ was a vector field on $(M,V_1)$, and hence commutes with $V_1$ (Proposition \ref{propCommuteVFs}) and so, by Corollary \ref{corCommuteIsoInvariant}, $V_1$ is invariant under the solution to $V_2$, that is, $V_1$ is invariant under $\gamma$. Thus $\gamma$ is a map in $\VF$, and the required commuting diagrams to be a solution hold since they hold in $\X$.

We can now show that $((C,0),c_1,c_0)$ is a curve object:
\begin{itemize}
	\item Uniqueness of solutions follows since it holds for $C$ in $\X$: by the above, any two solutions to a system in $\VF$ are solutions to a system in $\X$ and so must be equal.
	\item Completeness of $c_1$ holds since $c_1$ has a solution in $\X$, and hence by the above result also has a solution in $\VF$.  
	\item Commutativity of $c_1$ with itself is immediate since it holds in $\X$.
 \end{itemize}
\end{proof}

\begin{corollary}\label{corFLOWCurveObject}
If $(C,c_1,c_0)$ is a curve object in $\X$, then $((C,\pi_1),c_1,c_0)$ is a curve object in $\FLOW$. 
\end{corollary}
\begin{proof}
By Corollary \ref{corIsoCVF_Flow}, $\FLOW$ is isomorphic as a tangent category to $\CVF$. Therefore, it follows that $((C,\gamma(0)),c_1,c_0)$ is a curve object in $\FLOW$. But by Lemma \ref{lemmaZeroSolution}, $\gamma(0) = \pi_1$.
\end{proof}

As with vector fields in $\VF$ and $\FLOW$, flows in these tangent categories are also important:

\begin{proposition}\label{propFlowInVFandFLOW}
A flow $((M, \gamma_1),\gamma_2)$ in the tangent category $\FLOW$ is a pair of commuting flows $\gamma_1,\gamma_2$ in $\X$; a flow $((M,V),\gamma)$ in the tangent category $\CVF$ is a flow $\gamma$ in $\X$ such that $V$ is invariant under $\gamma$.
\end{proposition}
\begin{proof}
The ``moreover'' part of Proposition \ref{propVFCurveObject} tells us that the monoid structure of $C$ in $\X$ is the same as the monoid structure of $(C,0)$ in $\VF$ (and hence also of $(C,\pi_1)$ in $\FLOW$), and so the required equalities to be a flow are the same in all three tangent categories. The only thing that is required, then, is to understand what it means for the flows to be maps in $\FLOW$ and $\VF$.  

For $\FLOW$, first note that by the definition of the product structure in $\FLOW$ (see Corollary \ref{corIsoCVF_Flow})
	\[ (C,\pi_1) \times (M,\gamma_1) = (C \times M, \<\<\pi_0,\pi_1\>\pi_1,\<\pi_0,\pi_2\>\gamma_1\>) = (C \times M, \<\pi_1,\<\pi_0,\pi_2\>\gamma_1\>) \]
Thus, a map $\gamma_2$ in $\FLOW$ from $(C, \pi_1) \times (M,\gamma_1)$ to $(M,\gamma_1)$, is a map from $C \times M \to M$ so that the following diagram commutes:
\[
\bfig
	\square<1000,350>[C \times C \times M`C \times M`C \times M`M;\<\pi_1,\<\pi_0,\pi_2\>\gamma_1\>`1 \times \gamma_2`\gamma_2`\gamma_1]
\efig
\]
However, $\<\pi_1,\<\pi_0,\pi_2\>\gamma_1\> = (\tau \times 1)(1 \times \gamma_1)$ (where $\tau = \<\pi_1,\pi_0\>$) and so the above diagram is the same as asking that $\gamma_1$ and $\gamma_2$ commute. 

For the other part, if $((M,V),\gamma)$ is a flow in $\VF$, then it must be a map in $\VF$ from $(C \times M, 0 \times V)$ to $(M,V)$, and so must make the following diagram commute:
\[
\bfig
	\square<500,350>[C \times M`M`T(C \times M)`TM;\gamma`0 \times V`V`T(\gamma)]
\efig
\]
that is, $V$ must be invariant under $\gamma$.
\end{proof}

We can now complete the proof of Theorem \ref{thmCommFlows}. \\

\begin{proof} (of Theorem \ref{thmCommFlows})
By Corollary \ref{corCommuteIsoInvariant}, (ii), (iii), and (iv) are equivalent. However, since $\FLOW$ and $\CVF$ are isomorphic tangent categories, flows (with respect to their canonical curve objects) are equivalent, and by \ref{propFlowInVFandFLOW}, a flow in $\FLOW$ is a pair of commuting flows, while a flow in $\CVF$ is a vector field which is invariant under that flow. Thus, (i) is equivalent to (ii), and so all statements are equivalent.
\end{proof}

To sum up, we have the following table:
\begin{center}
\begin{tabular}{ c | c | c }
	& in $\CVF$ & in $\FLOW$ \\ \hline
vector field & Pair of commuting vector fields & Vector field invariant under a flow \\
flow & Vector field invariant under a flow & Pair of commuting flows
\end{tabular}
\end{center}
Items in each row are isomorphic since $\CVF$ and $\FLOW$ are isomorphic tangent categories. However, the columns are also isomorphic since in any given tangent category with curve object, complete vector fields and flows are isomorphic (Theorem \ref{thmCVFFlow}). This gives a different way to prove Theorem \ref{thmCommFlows}. This table could also be extended by considering, for example, vector fields in the tangent category $\sf{VF}(\sf{VF}(\X))$ (which would be a set of three vector fields which pairwise commute), or more generally, a vector field in the tangent category $\sf{VF}^{n}$, which would be a set of $n+1$ vector fields which pairwise commute.  We discuss this a bit further in the conclusion (section \ref{secConclusions}). \\
	
We conclude this section with the generalization of another important result on commuting vector fields.  
\begin{proposition}\label{propSumVFs}
Suppose that $V_1, V_2$ are complete vector fields with associated flows $\gamma_1, \gamma_2$. If $V_1$ and $V_2$ commute, then their sum $\<V_1,V_2\>+$ is also a complete vector field, with flow
	\[ C \times M \to^{\Delta \times 1} C \times C \times M \to^{1 \times \gamma_1} C \times M \to^{\gamma_2} M. \]
(Note that since $V_1, V_2$ commute, the above expression is also equal to $(\Delta \times 1)(1 \times \gamma_2)\gamma_1)$.
\end{proposition}
\begin{proof}
The initial condition is straightforward; thus, all we need to show is that the following diagram commutes:
\[
\xymatrix{ C \times M \ar[r]^{\Delta \times 1} \ar[d]_{c_1 \times 0} & C \times C \times M \ar[r]^{1 \times \gamma_1} & C \times M \ar[r]^{\gamma_2} & M \ar[d]^{\<V_1,V_2\>+} \\ 
		T(C \times M) \ar[rrr]_{T((\Delta \times 1)(1 \times \gamma_1)\gamma_2)} & & & TM}
\]
To prove this, we will make use of the result that ``the derivative is the sum of its partial derivatives'' \cite[Proposition 2.10]{sman3}. Applied to the map $f := (1 \times \gamma_1)\gamma_2$, this gives
	\[ T((1 \times \gamma_1)\gamma_2) = T(f) = \<(1 \times p0 \times p0)T(f),(p0 \times 1 \times p0)T(f),(p0 \times p0 \times 1)T(f)\>+ (\star). \]
Thus, the bottom left composite of the diagram is equal to the sum of the three components in $\star$, each pre-composed with $(c_1 \times 0)T(\Delta \times 1)$.  
Now, we have
	\[ (c_1 \times 0)T(\Delta \times 1) = \<\pi_0c_1,\pi_0c_1,\pi_10\>. \]
Pre-composing this with the first component of the sum in $\star$ thus gives
\begin{eqnarray*}
&  & \<\pi_0c_1,\pi_0c_1,\pi_10\>(1 \times p0 \times p0) T(f) \\
& = & \<\pi_0c_1, \pi_00,\pi_10\>T(1 \times \gamma_1)T(\gamma_2) \\
& = & \<\pi_0c_1,\<\pi_0,\pi_1\>\gamma_10\>T(\gamma_2) \mbox{ (by naturality of $0$)} \\
& = & (\Delta \times 1)(1 \times \gamma_1)(c_1 \times 0)T(\gamma_2) \\
& = & (\Delta \times 1)(1 \times \gamma_1)\gamma_2V_2 \mbox{ (by definition of $\gamma_2$)} 
\end{eqnarray*}
Thus, one component of the bottom-left composite of the original diagram is equal to the second component of the top-right composite. Pre-composing $(c_1 \times 0)T(\Delta \times 1)$ with the second component of the sum in $\star$ gives
\begin{eqnarray*}
&  & \<\pi_0c_1,\pi_0c_1,\pi_10\>(p0 \times 1 \times p0) T(f) \\
& = & \<\pi_00, \pi_0c_1, \pi_10\>T(1 \times \gamma_1)T(\gamma_2) \\
& = & \<\pi_00, (c_1 \times 0)T(\gamma_1)\>T(\gamma_2) \\
& = & \<\pi_00,\gamma_1 V_1\>T(\gamma_2) \mbox{ (by definition of $\gamma_1$)} \\
& = & (\Delta \times 1)(1 \times \gamma_1)(0 \times V_1)T(\gamma_2) \\
& = & (\Delta \times 1)(1 \times \gamma_1)\gamma_2 V_1 \mbox{ (since $\gamma_2$ is invariant under $V_1$)} 
\end{eqnarray*}
Thus the second component of the bottom-left composite of the original diagram is equal to the first component of the top-right composite.  Finally, we need to consider the last term in $\star$, pre-composed with $(c_1 \times 0)T(\Delta \times 1)$:
\begin{eqnarray*}
&  & \<\pi_0c_1,\pi_0c_1,\pi_10\>(p0 \times p0 \times 1)T(f) \\
& = & \<\pi_0 0, \pi_00, \pi_10\>T(f) \\
& = & (\Delta \times 1)f0 \mbox{ (by naturality of $0$)} 
\end{eqnarray*}
Thus, the last component of the bottom-left composite is a $0$ term, and so contributes nothing to the sum. Taken together, the calculations above show that the original diagram commutes, as required.
\end{proof}

\subsection{Reversing solutions}

As tangent categories do not necessarily have negatives, we have no guarantee that our solutions can be reversed; that is, that we can ``go backwards in time''.    However, if our tangent category has negatives ($-: T \to T$) and a solution for the system $(C,c_1-,c_0)$, then we can recover some of the classical results relating to reversing solutions.  

\begin{proposition}\label{groupStructureOfC}
Suppose $\X$ has negatives, and that the system $(C,c_1-,c_0)$ has a solution $\nc: C \to C$.  Then the commutative monoid $(C,\pc,c_0)$ is a (Abelian) group with inverse $\nc$.  
\end{proposition}

 
\begin{proof}
Consider the diagram
\[
\bfig
	\square<700,350>[C \times C`T(C \times C)`C`TC;c_1 \times c_1-`\pc`T(\pc)`0]
	\square(0,350)|arra|<700,350>[C`TC`C \times C`T(C \times C);c_1`\<1,\nc\>`T(\<1,\nc\>)`]
	\morphism(-700,700)<700,0>[1`C;c_0]
	\morphism(-700,700)<700,-350>[1`C \times C;\<c_0,c_0\>]
	\morphism(-700,700)|b|<700,-700>[1`C;c_0]
\efig
\]
All regions of the diagram are straightforward except for the bottom right one.  For this, consider the following calculation:
\begin{eqnarray*}
&   & (c_1 \times c_1-)T(\pc) \\
& = & (c_1 \times c_1-)\<(1 \times p0)T(\pc), (p0 \times 1)T(\pc)\>+\mbox{ (by \cite[Proposition 2.10]{sman3})} \\
& = & \<(c_1 \times c_1-p0)T(\pc),(c_1p0 \times c_1-)T(\pc)\>+ \\
& = & \<(c_1 \times 0)T(\pc),(0 \times c_1-)T(\pc)\>+ \mbox{ ($c_1$ and $c_1-$ are vector fields)} \\
& = & \<(c_1 \times 0)T(\pc),(0- \times c_1-)T(\pc)\>+ \mbox{ ($0$ is the identity for $+$)} \\
& = & \<(c_1 \times 0)T(\pc),(0 \times c_1)T(\pc)-\>+ \mbox{ (naturality of $-$)} \\
& = & \<\pc c_1, \pc c_1 -\>+  \mbox{ (Lemma \ref{lemmaPlusProps})} \\
& = & \pc c_1\<1,-\>+ \\
& = & \pc c_1 p0 \\
& = & \pc 0
\end{eqnarray*}

Thus we have shown that $\<1,\nc\>\pc$ solves the system $(C,0,c_0)$.  But by Lemma \ref{lemmaZeroSolution}, this system is also solved by $\pi_1 c_0 = !c_0$.  Hence by uniqueness, $\<1,\nc\>\pc = !c_0$, so that $\nc$ is an inverse for $\pc$.   
\end{proof}

In this setting, if we have a solution to a system with vector field $V$, then we can also solve a system with vector field the negative of $V$.  

\begin{proposition}
Suppose that $\X$ has negatives, and that $\nc$ is a solution to $(c_1-,c_0)$.  If $\gamma$ is a solution to $(V,g)$, then $(\nc \times 1)\gamma$ is a solution to $(V-,g)$. 
\end{proposition}
\begin{proof}
For the initial condition, since $c_0\nc = c_0$,
	\[ \<!c_0,1\>(\nc \times 1)\gamma = \<!c_0,1\>\gamma = g. \]
For the derivative condition, consider
\[
\bfig
	\square<700,350>[C \times X`T(C \times X)`M`TM;c_1 \times 0`\gamma`T\gamma`V]
	\square(700,0)<700,350>[T(C \times X)`T(C \times X)`TM`TM;- \times -``T\gamma`-]
	\square(0,350)/>`>`>`/<1400,350>[C \times X`T(C \times X)`C \times X`T(C \times X);c_1 \times 0`\nc \times 1`T(\nc \times 1)`]
\efig
\]

The top region commutes by definition of $\nc$, the bottom left by definition of $\gamma$, and the bottom right by naturality of $-$. 
\end{proof}

A standard result in differential geometry is that  flows give diffeomorphisms at each instant in time.   We can recover this result with the above assumptions:
\begin{proposition}
Suppose that $\X$ has negatives, and that $\nc$ is a solution to $(C,c_1-,c_0)$.  If $\gamma$ is a flow on $M$, then for any point $t: 1 \to C$, the map 
	\[ M \to^{\<!t,1\>} C \times M \to^{\gamma} M \]
is an isomorphism, with inverse
	\[ M \to^{\<!t\nc,1\>} C \times M \to^{\gamma} M. \]
\end{proposition}
\begin{proof}
This is a straightforward calculation which uses the two properties of a flow and the result that $\nc$ is an inverse for $\pc$:
\begin{eqnarray*}
&   & \<!t,1\>\gamma\<!t\nc,1\>\gamma \\
& = & \<!t,1\>\<!t\nc,\gamma\>\gamma \\
& = & \<!t,1\>\<!t\nc,1\>(1 \times \gamma)\gamma \\
& = & \<!t,1\>\<!t\nc,1\>(\pc \times 1)\gamma \mbox{ (by property (ii) of a flow)} \\
& = & \<!t, !t\nc,1\>(\pc \times 1)\gamma \\
& = & \<!t\<1,\nc\>\pc,1\>\gamma \\
& = & \<!c_0,1\>\gamma \mbox{ (by Proposition \ref{groupStructureOfC})} \\
& = & 1_M \mbox{ (by property (i) of a flow)}
\end{eqnarray*}
The other composite is similar.  
\end{proof}





\section{Linear completeness and differential curve objects}\label{secLinearSystems}

As noted earlier, not all vector fields have solutions. However, in most examples, a certain subset of vector fields are guaranteed solutions: the linear vector fields. In this section, we explore some of the consequences of assuming that linear systems have solutions. In particular, we shall show a somewhat surprising result: if the curve object has solutions for all linear systems and is itself a differential object, then there is a canonical multiplication operation on the curve object, and a bilinear action of the curve object on any differential bundle.

\subsection{Curve objects with solutions for linear systems}

\begin{definition}
Say that a curve object $(C,c_1,c_0)$ satisfies \textbf{linear completeness} if whenever $q: A \to M$ is a differential bundle and $(V^A,V^M)$ is a linear vector field on it with $V^M$ complete, then $V^A$ is also complete.  
\end{definition}
In the presence of pullbacks of differential bundles, one can simplify this axiom so that one only needs to consider linear vector fields $(V^A,V^M)$ where $V^M = c_1 \times 0$: see section \ref{section:axiomSimplification}.  

\begin{example}
\begin{enumerate}[(a)]
	\item In the tangent categories of smooth Euclidean spaces and smooth manifolds, the canonical curve object $\mathbb{R}$ satisfies linear completeness. A differential bundle is, in particular, a vector bundle, and the solution to a linear vector field is the matrix exponential (see \cite[Prop. 20.3]{lee}) in each fibre.
	\item In the tangent category of analytic maps between $\mathbb{C}^n$'s, the curve object $\mathbb{C}$ satisfies linear completeness, with solutions again being matrix exponentials (see \cite[pg. 99]{hille}).  
	\item In a model of SDG, the curve object $D_{\infty}$ satisfies linear completeness, as it has solutions for \textbf{all} vector fields.  
\end{enumerate}
\end{example}

\begin{remark}
In the tangent category of polynomial functions between Euclidean spaces, $\mathbb{R}$ does not satisfy linear completeness, as the linear vector field $x \mapsto \<x,x\>$ (on $\mathbb{R}$) has no polynomial solution, as its unique solution is the non-polynomial exponential flow $E(t,x) = xe^t$.  
\end{remark}

\begin{remark}\label{rmkParallelTransport}
The linear completeness axiom is a modification of an axiom first given in \cite[Definition 5.19]{connections}, where an initial version of the notion of curve object was first defined; the modification is due to Matthew Burke and Ben MacAdam.  The purpose of the notion in that paper was to show how a curve object allows one to define parallel transport for a connection in the setting of a tangent category. The linear completeness requirement given here also allows one to define parallel transport, as the vector field $F$ defined in the proof of \cite[Theorem 5.20]{connections} is a linear vector field over $c_1$ (which is complete by assumption).  
\end{remark}

Our goal in this section is to explore some consequences of the linear completeness axiom; as such, throughout the rest of this section, we assume that $C$ is a curve object with linear completeness.  We first begin with some basic results about solutions of linear systems in general before specifying to some particular systems.

\begin{proposition}\label{propLinearSol}
Suppose that $q: A \to M$ is a differential bundle (with lift map $\lambda: A \to TA$), $(V^A,V^M)$ is a linear vector field on it, and $\gamma: C \times A \to A$ is a solution to $V^A$. Then $\gamma$ is ``linear in $A$'' in the sense that the following diagram commutes:
\[
\bfig
	\square<500,350>[C \times A`A`T(C \times A)`TA;\gamma`0 \times \lambda`\lambda`T(\gamma)]
\efig
\]
\end{proposition}
\begin{proof}
We claim that both composites solve the system $(TA,T(V^A)c, \lambda)$. By Lemma \ref{lemmaTFcSoln2}, since $\gamma$ solves $V^A$, $(0 \times 1)T(\gamma)$ solves $T(V^A)c$, so by lemma \ref{lemmaParamFromFlow}, $(0 \times \lambda)T(\gamma)$ solves $(TA,T(V^A)c,\lambda)$. 

Now we want to show that $\gamma \lambda$ solves it as well; that is, we need to show that the following diagram commutes:
\[
\bfig
	\square<600,350>[C \times A`T(C \times A)`TA`T^2A;c_1 \times 0`\gamma \lambda`T(\gamma\lambda)`T(V^A)c]
	\morphism(-600,350)<600,0>[A`C \times A;\<!c_0,1\>]
	\morphism(-600,350)|b|<600,-350>[A`TA;\lambda]
\efig \]
The unit condition is direct from the definition of $\gamma$: $\<!c_0,1\>\gamma \lambda = 1_A \lambda = \lambda$. For the derivative condition, consider
\begin{eqnarray*}
&  & (c_1 \times 0)T(\gamma)T(\lambda) \\
& = & \gamma V^A T(\lambda) \mbox{ (by definition of $\gamma$)} \\
& = & \gamma \lambda T(V^A)c \mbox{ (since $V^A$ is linear)} \\
\end{eqnarray*}
as required. Thus, since $(0 \times \lambda)T(\gamma)$ and $\gamma \lambda$ solve the same system, they are equal.  
\end{proof}
Note that such $\gamma$ will not, in general, be linear in $C$ (if $C$ is itself a differential object or bundle). Also note that $\gamma$ may not be a \emph{bundle} map; that is, the following diagram may not necessarily commute: 
\[
\bfig
	\square<500,350>[C \times A`A`C \times M`M;\gamma`1 \times q`q`\pi_1]
\efig
\]
(for example, the linear vector field defined in parallel transport (see \cite[Theorem 5.20]{connections}) does not satisfy this commutativity).  However, if the vector field is over the 0 vector field, then it is a bundle map. This is true for any vector field over $0$, not just linear vector fields:

\begin{proposition}\label{propBundleSol}
Suppose that $q: A \to M$ is any map, and $V^A$ is a vector field over $0_M: M \to TM$; that is the following diagram commutes:
\[
\bfig
	\square<500,350>[A`TA`M`TM;V^A`q`T(q)`0_M]
\efig
\]
Then if $\gamma: C \times A \to A$ is a solution of $V^A$, then $(\gamma,\pi_1)$ is a bundle map from $(1 \times q): C \times A \to C \times M$ to $q: A \to M$; that is, the following diagram commutes:
\[
\bfig
	\square<500,350>[C \times A`A`C \times M`M;\gamma`1 \times q`q`\pi_1]
\efig
\]
\end{proposition}
\begin{proof}
We claim both composites solve the system $(M,0_M,q)$. By Lemma \ref{lemmaZeroSolution}, $\pi_1: C \times M \to M$ solves $(M,0_M,1_M)$, so by Lemma \ref{lemmaParamFromFlow}, $(1 \times q)\pi_1$ solves the system $(M,0_M,q)$.  

For the other composite $(\gamma q)$, by definition of $\gamma$, $\<!c_0,1\>\gamma q = q$, so the initial condition is satisfied. For the derivative condition,
\begin{eqnarray*}
&  & (c_1 \times 0)T(\gamma)T(q) \\
& = & \gamma V^A T(q) \mbox{ (by definition of $\gamma$)} \\
& = & \gamma q 0_M \mbox{ (since $V^A$ is over $0_M$)} 
\end{eqnarray*}
as required. Thus, by uniqueness of solutions, $(1 \times q)\pi_1 = \gamma q$.

\end{proof}

Combining the two previous results, we have:
\begin{corollary}\label{corSolutionLinBundleMap}
If $q: A \to M$ is a diffferential bundle, $(V^A,0_M)$ is a linear vector field on it, and $\gamma: C \times A \to A$ is the solution of $V^A$, then the pair $(\gamma,\pi_1)$ is a linear bundle morphism from the differential bundle $(1 \times q): C \times A \to C \times M$ (with lift $0 \times \lambda: C \times A \to T(C \times A)$) to the differential bundle $q: A \to M$.
\end{corollary}

\subsection{Exponential functions}\label{secExponentials}

Every differential bundle has a canonical linear vector field associated to it:

\begin{proposition}\label{propEVFLinear}
If $q: A \to M$ is a differential bundle, and we define $\evf{A}$ as the composite
	\[ A \to^{\<1,1\>} A_2 \to^{\mu} TA \]
then $(\evf{A},0_M)$ is a linear vector field on $q: A \to M$.  
\end{proposition}
\begin{proof}
It is a vector field since $\mu p = \pi_0$ (see \cite[Lemma 2.9]{diffBundles} and the comments at the end of Section \ref{secPrelims}):
	\[ \evf{A}p = \<1,1\>\mu p = \<1,1\>\pi_0 = 1. \]
It is over $0_M$ since $\mu T(q) = \pi_0 q 0$ (\cite[Lemma 2.9, 2.10]{diffBundles}):
	\[ \evf{A}T(q) = \<1,1\>\mu T(q) = \<1,1\>\pi_0 q0 = q0. \]
For linearity, we need to show that $\evf{A}T(\lambda)c = \lambda T(\evf{A})$:
\begin{eqnarray*}
&  & \evf{A}T(\lambda)c \\
& = & \<1,1\>(0 \times \lambda)T(\sigma)T(\lambda)c \\
& = & \<0,\lambda\>T((\lambda \times \lambda)T(\sigma))c \mbox{ (additivity of $\lambda$)} \\
& = & \<0T(\lambda), \lambda T(\lambda) \>T^2(\sigma)c \\
& = & \<\lambda 0 c, \lambda \ell c\>T^2(\sigma) \mbox{ (naturality of $c$) and coherence of $\lambda$)} \\
& = & \<\lambda T(0), \lambda \ell\>T^2(\sigma) \mbox{ (coherence of $\ell$ and $c$)} \\
& = & \<\lambda T(0), \lambda T(\lambda) \>T^2(\sigma) \\
& = & \lambda T(\<0,\lambda\>T(\sigma)) \\
& = & \lambda T(\evf{A})
\end{eqnarray*}
as required. 
\end{proof}
Following standard terminology, we shall refer to this vector field as the \textbf{Euler vector field on $q: A \to M$}.  

\begin{example}
In local coordinates on a differential bundle in smooth manifolds, $\evf{TM}$ is the map
	\[ (x,v) \mapsto (x,v,0,v). \]
\end{example}

Since $(\evf{A},0_M)$ is a linear vector field, and $0_M$ always has a solution (Lemma \ref{lemmaZeroSolution}), the linear completeness axiom ensures that $\evf{A}$ is complete.  

 \begin{definition}\label{defnExpFunctions}
If $q: A \to M$ is a differential bundle, let $\expf{A}$ denote the (unique) flow of the Euler vector field $\evf{A}$, so that the following diagram commutes:
\[
\bfig
	\square/>`.>`.>`>/<600,350>[C \times A`T(C \times A)`A`TA;c_1 \times 0`\expf{A}`T(\expf{A})`\evf{A}]
	\morphism(-600,350)<600,0>[A`C \times A;\<!c_0,1\>]
	\morphism(-600,350)|b|<600,-350>[A`A;1_A]
\efig
\]
We will refer to $\expf{A}$ as the \textbf{exponential function associated to $q: A \to M$}.  
\end{definition}

\begin{example}
If $A = \mathbb{R}^n$ (considered as a differential bundle over $1$), then $\evf{A}(x) = \<x,x\>$, and the flow $\expf{A}: \R \times \R^n \to \R^n$ is simply the ordinary exponential flow:
	\[ \expf{A}(t,x) = e^t \cdot x. \]
\end{example}

\begin{corollary}\label{corExp}
The pair $(\expf{A},\pi_1)$ is a linear bundle morphism from the product bundle $(1 \times q): C \times A \to C \times M$ to $q: A \to M$.
\end{corollary}
\begin{proof}
This is immediate from Proposition \ref{propEVFLinear} and Corollary \ref{corSolutionLinBundleMap}.  
\end{proof}

Linear maps between differential bundles preserve these exponential functions.
\begin{proposition}\label{propLinPresExp}
Suppose that $q: A \to M$ and $q': A' \to M'$ are differential bundles and $(f,g)$ is a linear bundle morphism between them.  Then $f$ preserves their associated exponential maps; that is, the following diagram commutes:
\[
\bfig
	\square<500,350>[C \times A`A`C \times B`B;\expf{A}`1 \times f`f`\expf{B}]
\efig
\]
\end{proposition}
\begin{proof}
Since $f$ is linear, by \cite[Lemma 2.17]{diffBundles}, $f$ also preserves the associated $\mu$ maps for the bundles; thus, it also preserves their associated Euler vector fields:
	\[ \evf{A} T(f) = \<1,1\>\mu^A T(f) = \<1,1\>(f \times f)\mu^B = f\<1,1\>\mu^B = f \evf{B}. \]
Thus, by Proposition \ref{propVFMaps}, $f$ preserves their associated flows, ie., their associated exponential maps.
\end{proof}


\subsection{Differential curve objects}

While the results of the previous section are true for any curve object $(C,c_1,c_0)$ we will see that we can derive much more if we assume $C$ is also a differential object. Recall \cite[Definition 3.1]{diffBundles} that a differential object is a commutative monoid $(A, \sigma, \zeta)$ with a map $\hat{p}: TA \to A$ so that $TA$ is a product of $A$ with itself (with projections $\hat{p}, p_A$). A differential object then also has an associated map $\lambda: A \to TA$ making $A$ a differential bundle over $1$ (see \cite[Proposition 3.4]{diffBundles}). 

\begin{definition}\label{defnDiffCurveObject}
A \textbf{differential curve object} is a curve object $(C,c_1,c_0)$ satisfying linear completeness, which also has the structure of a differential object $(\hat{p}, \sigma, \zeta)$, such that
\begin{enumerate}[(i)]
	\item $c_0 = \zeta$, 
	\item if $u$ is defined to be the point
		\[ 1 \to^{c_0} C \to^{c_1} TC \to^{\hat{p}} C \]
	then $c_1 \hat{p} = !u$. 
\end{enumerate}
\end{definition}

We shall see below (Corollary \ref{cor:altDiffCurve}) that the definition can be reformulated to remove the curve object requirements that $c_1$ be self-commutative and complete, as these will follow automatically from the rest of the structure.  

\begin{example}
\begin{enumerate}[(a)]
	\item In the tangent categories of smooth functions between Euclidean spaces and smooth manifolds, $\R$ is a differential curve object (with $u$ the multiplicative unit $1$).  
	\item In the tangent category of analytic functions between $\mathbb{C}^n$'s, $\mathbb{C}$ is a differential curve object (again, with $u = 1$). 
\end{enumerate}
\end{example}

\begin{remark}
In a model of SDG, $D_{\infty}$ is \emph{not} a differential curve object, as it is not a differential object.
\end{remark}

\begin{lemma}\label{lemma_U}
If $C$ is a differential curve object then 
\begin{enumerate}[(i)]
	\item $c_1 = \<1_c,!u\>$.  
	\item $c_0c_1 = u\lambda$.  
\end{enumerate}
\end{lemma}
\begin{proof}
For (i), definition of a differential curve object, $c_1 \hat{p} = !u$, and since $c_1$ is a vector field, $c_1p = 1_c$. Thus the result follows since $TC$ is a product with projections $(p,\hat{p})$.  

For (ii), since these are both maps into $TC$, it suffices to check their equality when post-composed by the projections ($\hat{p}$ and $p$). Indeed,	
	\[ u \lambda \hat{p} = u = c_0c_1\hat{p} \]
and
	\[ c_0c_1p = c_0 = \zdc = u \lambda p. \]
\end{proof}

\begin{proposition}\label{prop:plusIsFlow}
If $C$ is a differential curve object then $\sigma$ is the flow for $c_1$.
\end{proposition}
\begin{proof}
The initial condition is true since $c_0 = \zeta$:
	\[ \<!c_0,1_C\>\sigma = \<\zeta,1_C\>\sigma = 1_C. \]
For the derivative condition, by Proposition \ref{propDiffObjectSolns}, it suffices to prove that 
	\[ (c_1 \times 0)T(\sigma)\hat{p} = \sigma c_1\hat{p}. \]
By (ii) for a differential curve object, 
	\[ \sigma c_1\hat{p} = \sigma !u = !u \]
while 
\begin{eqnarray*}
&  & (c_1 \times 0)T(\sigma)\hat{p} \\
& = & (c_1 \times 0)\hat{p}\sigma \mbox{ ($\hat{p}$ is additive)} \\
& = & (c_1\hat{p}\times 0\hat{p})\sigma \\
& = & (!u \times \zeta )\sigma \mbox{ ($\hat{p}$ is additive and (ii) for a differential curve object)} \\
& = & !u
\end{eqnarray*}
as required.  
\end{proof}

\begin{corollary}\label{cor:altDiffCurve}
To define a differential curve object, it suffices to give a differential object $(C,\hat{p},\sigma,\zeta)$, with a point $u: 1 \to C$ such that if $c_0 := \zeta: 1 \to C$ and $c_1 := \<1_C,!u\>: C \to TC$, then the system $(C,c_1,c_0)$ satisfies the preinitial and linear completeness axioms.  
\end{corollary}
\begin{proof}
By definition, $(c_1,c_0)$ satisfy axioms (i) and (ii) for a differential curve object. Using the same proofs as above, it then follows that $\sigma$ is a complete solution for $c_1$. Moreover, by Corollary \ref{corFlipCondition}, the requirement $c_1T(c_1)c = c_1T(c_1)$ is automatic. Thus, $C$ is a curve object and hence a differential curve object.
\end{proof}

We'll next see that on a differential curve object, we can view the exponential function $\expf{C}: C \times C \to C$ as the derivative of a function from $C$ to itself.  Note that for a differential object $A$, the associated map $\mu$ has type $\mu: A \times A \to TA$, and is an isomorphism. We will use this fact in some of the proofs below.

\begin{definition}
Let $e: C \to C$ denote the solution to the system $(C,\evf{C},u)$, so that the following diagram commutes:
\[
\bfig
	\square/>`.>`.>`>/<600,350>[C`TC)`C`TC;c_1`e`T(e)`\evf{A} = \<1,1\>]
	\morphism(-600,350)<600,0>[1`C;c_0]
	\morphism(-600,350)|b|<600,-350>[1`C;u]
\efig
\]
\end{definition} 

\begin{example}
In smooth manifolds, with $C = \mathbb{R}$, $e$ is the ordinary exponential function: $e(t) = e^t$. 
\end{example}

In the following proposition, we use the language of Cartesian differential categories, writing $0$ and $+$ for the unit and addition operations on $C$, and $D(f)$ for the derivative of a map $f$ (that is, $D(f) = T(f)\hat{p}$).  

\begin{proposition}\label{prop_eProperties}
The map $e$ satisfies the following equations:
\begin{enumerate}[(i)]
	\item $e = \<1,!u\>\expf{C}$.
	\item $0e = u$.
	\item $\<1,!u\>D(e) = c_1D(e) = e$.  
	\item $\<0,1\>D(e) = 1_C$.
	\item $\<\pi_0,!u,\pi_1,0\>D^2(e) = D(e)$.  
	\item $D(e) = \expf{C}$. 
\end{enumerate}
\end{proposition}
\begin{proof}
\begin{enumerate}[(i)]
	\item Since $\expf{C}$ is the flow of $\evf{C}$, this is immediate from Lemma \ref{lemmaParamFromFlow}.  
	\item This is the initial condition requirement for a solution to the system $(\evf{C},u)$.
	\item This is the equivalent form of the derivative condition for the system $(\evf{C},u)$ (using Proposition \ref{propDiffObjectSolns}). 
	\item To prove this, we will show that $\<0,1\>D(e)$ solves the system $(C,c_1,c_0)$ (which is also solved by $1_C$). For the initial condition,
		\[ c_0\<0,1\>D(e) = \<0,0\>D(e) = 0 = c_0 \]
	using CD.2. For the derivative condition, we need to show that the following diagram commutes:
	\[
	\bfig
		\square(0,350)/>`>`>`/<500,350>[C`TC`C \times C`T(C \times C);c_1`\<0,1\>`T(\<0,1\>)`]
		\square/`>`>`>/<500,350>[C \times C`T(C \times C)`C`TC;`D(e)`T(D(e))`c_1]
	\efig
	\]
	Equivalently, after post-composing with $\hat{p}$, this is equivalent to asking that
	\[ c_1 T(\<0,1\>)D^2(e) = \<0,1\>D(e)!u = !u. \]		
	Expanding $c_1 T(\<0,1\>)D^2(e)$, we get
	\begin{eqnarray*}
	&  & c_1 \<\pi_0\<0,1\>,D(\<0,1\>)\>D^2(e) \\
	& = & c_1\<0,\pi_0,0,\pi_1\>D^2(e) \\
	& = & c_1\<0,\pi_1\>D(e) \mbox{ (by CD.6)} \\
	& = & \<0,!u\>D(e) \\
	& = & 0 \<1,!u\>D(e) \\
	& = & 0e \mbox{ (by iii)} \\
	& = & !u \mbox{ (by ii)}
	\end{eqnarray*}
	Thus, by uniqueness of solutions, $\<0,1\>D(e) = 1_C$.

	\item By (iii), we have $\<1,!u\>D(e) = e$. Differentiating both sides of this equation gives the required result:
		\[ D(\<1,!u\>D(e)) = \<\pi_0,!u,\pi_1,0\>D^2(e) = D(e). \]
	\item By uniqueness of solutions, it suffices to prove that $D(e)$ is also the solution to the system $(C,\evf{C},1_C)$. The initial condition is $\<0,1\>D(e) = 1_c$, which was proven in (iv). For the derivative condition, we need to show $(c_1 \times 0)D^2(e) = D(e)$. Indeed,
	\begin{eqnarray*}
	&  & (c_1 \times 0)D^2(e) \\
	& = & \<\pi_0,\pi_1,!u,0\>D^2(e) \\
	& = & \<\pi_0,!u,\pi_1,0\>D^2(e) \mbox{ (by CD.7)} \\
	& = & D(e) \mbox{ (by (v))} 
	\end{eqnarray*}
	Thus $D(e)$ solves the same system as $\evf{C}$, and so they are equal. 
\end{enumerate}
\end{proof}

Of course, one of the key properties of the exponential function is $e^{a+b} = e^a \cdot e^b$. Since the derivative operation $D$ implicitly involves multiplication, one way to express this is the following result:
\begin{proposition}\label{ePlus}
The following diagram commutes:
\[
	\xymatrix{ C \times C \ar[r]^{+} \ar[d]_{1 \times e} & C \ar[d]^{e} \\
	C \times C \ar[r]_{D(e)} & C }
\]
\end{proposition}
\begin{proof}
We will show these are equal by showing they both solve the system $(C,\evf{C},e)$. Since $\expf{C}$ solves $(\evf{C},1_C)$, by Lemma \ref{lemmaParamFromFlow}, $(1 \times e)\expf{C}$ solves $(\evf{C},e)$. However, by Proposition \ref{prop_eProperties}.vi, $D(e) = \expf{C}$, so indeed $(1 \times e)D(e)$ solves $(C,\evf{C},e)$.  

Now we want to show $+e$ also solves this system. The initial condition is immediate by unitality of $+$: $\<0,1\>+e = e$. For the derivative condition, we need to show that $(c_1 \times )D(+e) = +e$. Indeed:
\begin{eqnarray*}
&  & (c_1 \times 0)D(+e) \\
& = & \<\pi_0,\pi_1,!u,0\>\<\pi_0+, \pi_1+\>D(e) \mbox{ (by linearity of $+$)} \\
& = & \<+,!u\>D(e) \mbox{ (by unitality of $+$)} \\
& = & +\<1,!u\>D(e) \\
& = & +e \mbox{ (by Proposition \ref{prop_eProperties}.iii)}
\end{eqnarray*}
Thus $+e$ solves the same system as $(1 \times e)D(e)$, and so they are equal.  
\end{proof}

This is related to recent work on exponential functions in Cartesian differential categories. The following are from \cite{lemayExponential}:

\begin{definition}\label{defnDiffExponential}
In a Cartesian differential category, a \textbf{differential exponential map} is a map $e: A \to A$ so that the following diagrams commute:
 \[ \xymatrix{A \ar[r]^{\<0,1\>} \ar[dr]_{1_A} & A \times A \ar[d]^{D(e)} & A \times A \ar[r]^-{1 \times e} \ar[d]_-{+} & A \times A \ar[d]^-{D(e)} \\ & A & A \ar[r]_-{e} & A} \]
A \textbf{differential exponential rig} consists of a commutative monoid $(A,\bullet: A \times A \to A, u: 1 \to A)$ together with a map $e: A \to A$ so that:
\begin{enumerate}[(i)]
	\item $\bullet$ is linear in each variable (that is, $(A, \bullet, u)$ is a \emph{differential} rig);
	\item the following diagrams commute:
	\[ \xymatrix{A \times A \ar[r]^{e \times 1} \ar[dr]_{D(e)} & A \times A \ar[d]^{\bullet} & 1 \ar[dr]_-{u} \ar[r]^-{0} & A \ar[d]^-{e} & A \times A \ar[d]_-{+} \ar[r]^-{e \times e} & A \times A \ar[d]^-{\bullet} \\ 
	& A & & A & A \ar[r]_-{e} & A } \]
\end{enumerate}

\end{definition}

The following is proved in that paper:

\begin{proposition}\label{propExpEquivalence}
There is a bijective correspondence between differential exponential maps and differential exponential rigs; given a differential exponential pair $(A,e)$, one defines a rig structure by $u := 0e: 1 \to A$ and $\bullet$ by
	\[ A \times A \to^{\<0,1\> \times \<0,1\>} A \times A \times A \times A \to^{D^2(e)} A \]
\end{proposition}

For the full proof, see \cite[Proposition 4.5]{lemayExponential}. Here, we offer some intuition about why this works. In the standard case, we have $e(t) = e^t$. Its first directional derivative is then the map $D(e)(t,v) = ve^t$. Its second directional derivative is
	\[ D^2(t,v,t',v') = v'e^t + t've^t, \]
so that
	\[ D^2(0,v,t',0) = t'v, \]
in other words, ordinary multiplication. Thus, the multiplication of real numbers can be recovered from the second derivative of the exponential function. The above result generalizes this to an arbitrary Cartesian differential category. Associativity of this action is the trickiest part to prove, but involves considering $D^3$, where the multiplication of three variables appears.  

Applying this result to our setting, we get:
\begin{corollary}
The pair $(C,e)$ is a differential exponential pair (in the Cartesian differential category of differential objects of $\X$) and thus $C$ acquires the structure of a differential exponential rig, with unit and multiplication as above.
\end{corollary}
\begin{proof}
Proposition \ref{prop_eProperties}.iv and Proposition \ref{ePlus} showed that $(C,e)$ is a differential exponential rig, and so Proposition \ref{propExpEquivalence} gives $C$ the structure of a differential exponential rig.  
\end{proof}

Later (see Theorem \ref{thmActions}) we shall give an alternative derivation of the structure of this differential exponential rig using the properties of a curve object.

\subsection{Differential bundle actions}

Throughout the rest of this section, we assume $C$ is a differential curve object. We will now consider how to use the exponential flow to define a bilinear action of $C$ on any differential bundle. On vector bundles in the tangent category of smooth manifolds, this will be the ordinary scalar action of $\mathbb{R}$ in each fibre of the vector bundle.

\begin{lemma}
If $q: A \to M$ is a differential bundle, then the composite
	\[ C \times A \to^{\lambda \times 0} T(C \times A) \to^{T(\expf{A})} TA \]
lands in the vertical part of $TA$; that is,
	\[ (\lambda \times 0)T(\expf{A})T(q) = (\lambda \times 0)T(\expf{A})pq0. \]
\end{lemma}
\begin{proof}
Using Corollary \ref{corExp},
\[ (\lambda \times 0)T(\expf{A})T(q) = (\lambda \times 0)T(\expf{A} q) = (\lambda \times 0)T(\pi_1 q) = \pi_1 0 T(q) = \pi_1 q 0. \]
On the other hand,
\begin{eqnarray*}
&  & (\lambda \times 0)T(\expf{A})pq0 \\
& = & (\lambda p \times 0p) \expf{A}q0 \mbox{ (naturality of $p$}) \\
& = & (!c_0 \times 1)\expf{A}q0 \\
& = & \<!c_0,\pi_1\>\expf{A}q0 \\
& = & \pi_1 q 0 \mbox{ (by definition of $\expf{A}$)}
\end{eqnarray*}
as required.
\end{proof}

Thus, by \cite[Lemma 2.10.ii]{diffBundles}, we can take the bracket of this map to get a map from $C \times A \to A$.  

\begin{definition}
Define $\action{A}: C \times A \to A$ as the map $\{(\lambda \times 0)T(\expf{A})\}: C \times A \to A$.  
\end{definition}

\begin{example}
If $A$ is $\mathbb{R}^n$, then as above, $\expf{A}(t,x) = e^t \cdot x$, and so
	\[ T(\expf{A})(t,x,t',x') = (e^t \cdot x, t'e^t \cdot x + e^t \cdot x'). \]
Pre-composing with $(\lambda \times 0)$ sets $t = x' = 0$, and taking the bracket simply keeps the derivative component, so in this case we have
	\[ \action{A}(x,t') = t' \cdot x, \]
in other words, the ordinary scalar action of $\mathbb{R}$ on $\mathbb{R}^n$.
\end{example}

Our goal is to show that $\action{A}$ is a bilinear action of $C$ on $A$.  

\begin{lemma}\label{lemmaBundleActionUnital}
$u: 1 \to C$ is a unit for the map $\action{A}: C \times A \to A$; that is,
	\[ \<!u,1\>\action{A} = 1_A. \]
\end{lemma}
\begin{proof}
Consider
\begin{eqnarray*}
&   & \<!u,1\>\action{A} \\
& = & \<!u,1\>\{(\lambda \times 0)T(\expf{A})\} \\
& = & \{ \<!u \lambda, 0\>T(\expf{A})\} \\
& = & \{ \<!c_0 c_1, 0\>T(\expf{A})\} \mbox{ (by lemma \ref{lemma_U})} \\
& = & \{ \<!c_0,1\>(c_1 \times 0)T(\expf{A})\} \\
& = & \{ \<1,1\>\mu \} \mbox{ (by definition of $\expf{A}$)} \\
& = & \<1,1\>\pi_1 \mbox{ (see \cite[Lemma 2.12.vii]{diffBundles})} \\
& = & 1
\end{eqnarray*} 
as required.  
\end{proof}

\begin{proposition}\label{propActionLinearInA}
(Linearity of $\action{A}$ in $A$) If $q: A \to M$ is a differential bundle, then the pair $(\action{A},\pi_1)$ is a linear bundle morphism from the product bundle $(1 \times q): C \times A \to C \times M$ to $q: A \to M$.
\end{proposition}
\begin{proof}
We will use \cite[Proposition 2.8]{connections}. To use this, we will first prove that the pair $((\lambda \times 0)T(\expf{A}),\pi_1 0)$ is a linear bundle morphism from $(1 \times q): C \times A \to C \times M$ to $T(q): TA \to TM$ (with associated lift map $T(\lambda)c: TA \to T^2A$). This pair is a bundle morphism since by Corollary \ref{corExp},
	\[ (\lambda \times 0)T(\expf{A})T(q) = (\lambda \times 0)T(\pi_1q) = \pi_1 0 T(q) = \pi_1 q0 = (1 \times q)\pi_1 0. \]
To show that it is linear, we need to show that the following diagram commutes:
\[
\bfig
	\square<900,400>[C \times A`TA`T(C \times A)`T^2A;(\lambda \times 0)T(\expf{A})`0 \times \lambda`T(\lambda)c`T((\lambda \times 0)T(\expf{A}))]
\efig
\]
Indeed, we have
\begin{eqnarray*}
&   & (\lambda \times 0)T(\expf{A})T(\lambda)c \\
& = & (\lambda \times 0)T(\expf{A}\lambda)c \\
& = & (\lambda \times 0)T((0 \times \lambda)T(\expf{A}))c \mbox{ (by linearity of $\expf{A}$ in $A$: Proposition \ref{propLinearSol})} \\
& = & (\lambda T(0) \times 0 T(\lambda)) T^2(\expf{A})c \\
& = & (\lambda T(0) c \times \lambda 0 c) T^2(\expf{A}) \mbox{ (naturality of $0$ and $c$)} \\
& = & (\lambda 0 \times \lambda T(0)) T^2(\expf{A}) \\
& = & (0 T(\lambda) \times \lambda T(0))T^2(\expf{A}) \\
& = & (0 \times \lambda)T((\lambda \times 0)T(\expf{A}))
\end{eqnarray*}

Thus, by \cite[Proposition 2.8]{connections}, the pair $(\{(\lambda \times 0)T(\expf{A})\}, \pi_1 0p)$ is a linear bundle morphism from $(1 \times q): C \times A \to C \times M$ to $q: A \to M$. However, this pair equals $(\action{A},\pi_1)$, as required.
\end{proof}

\begin{proposition}\label{propActionLinearInC}
(Linearity of $\action{A}$ in $C$) The pair $(\action{A},q)$ is a linear bundle morphism from $\pi_1: C \times A \to A$ (the pullback bundle of $C$ along $!: A \to 1)$ to $q: A \to M$.  
\end{proposition}
\begin{proof}
We will use \cite[Proposition 2.9]{connections}. To use this result, we will first prove that the pair $((\lambda \times 0)T(\expf{A}),1)$ is a linear bundle morphism from $\pi_1: C \times A \to A$ to the bundle $p: TA \to A$ (whose lift is $\ell: TA \to T^2A$). It is a bundle morphism since
	\[ (\lambda \times 0)T(\expf{A})p = (\lambda p \times 0p)\expf{A} = \<!c_0,\pi_1\>\expf{A} = \pi_1\<!c_0,1\>\expf{A}= \pi_1. \]
For linearity, we need to show that the following diagram commutes:
\[
\bfig
	\square<900,400>[C \times A`TA`T(C \times A)`T^2A;(\lambda \times 0)T(\expf{A})`\lambda \times 0`\ell`T((\lambda \times 0)T(\expf{A}))]
\efig
\]
Indeed, we have
\begin{eqnarray*}
&  & (\lambda \times 0)T(\expf{A})\ell \\
& = & (\lambda \ell \times 0 \ell) T^2(\expf{A}) \mbox{ (naturality of $\ell$)} \\
& = & (\lambda T(\lambda) \times 0 T(0)) T^2(\expf{A}) \mbox{ (coherence of $\lambda$ and $\ell$)} \\
& = & (\lambda \times 0) T((\lambda \times 0)T(\expf{A}))
\end{eqnarray*}
Thus, by \cite[Proposition 2.9]{connections}, the pair $(\{(\lambda \times 0)T(\expf{A})\},q)$ is a linear bundle morphism from $\pi_1: C \times A \to A$ to $q: A \to M$. But this pair is $(\action{A},q)$, as required.  
\end{proof}

\begin{corollary}\label{corMuLinearity}
$\action{A}$ also satisfies the following form of linearity:
\[
\bfig
	\square<1200,400>[C \times C \times A`A_2`T(C \times A)`TA;\<\<\pi_0,\pi_2\>\action{A},\<\pi_1,\pi_2\>\action{A}\>`\mu \times 0`\mu`T(\action{A})]
\efig
\]
\end{corollary}
\begin{proof}
This is simply the result of applying \cite[Lemma 2.17]{diffBundles} to the linearity of $\action{A}$ in $C$.
\end{proof}

Using some of the above results, we can show that the action $\action{A}$ is itself part of a solution to a linear system.  

\begin{proposition}\label{propActionSolution}
The pairing $\<\action{A},\pi_1\>: C \times A \to A_2$ is the solution to the system $(A_2,\<\mu,\pi_10\>,\<q\zeta,1\>)$; that is, the following diagram commutes:
\[
\bfig
	\square<750,350>[C \times A`T(C \times A)`A_2`T(A_2);c_1 \times 0`\<\action{A},\pi_1\>`T(\<\action{A},\pi_1\>)`\<\mu,\pi_10\>]
	\morphism(-750,350)<750,0>[A`C \times A;\<!c_0,1\>]
	\morphism(-750,350)|b|<750,-350>[A`A_2;\<q \zeta, 1\>]
\efig
\]
\end{proposition}
\begin{proof}
The linearity of $\action{A}$ in $C$ tells us that $\<!c_0,1\>\action{A} = q \zeta$, and so the initial condition is satisfied.

For the derivative condition, consider
\begin{eqnarray*}
&  & (c_1 \times 0)T(\action{A}) \\
& = & (\<1,!u\>\mu \times 0)T(\action{A}) \mbox{ (by Lemma \ref{lemma_U} and Remark \ref{rmkMuOnDiffObject})} \\
& = & (\<\pi_0, !u\> \times 1)(\mu \times 0)T(\action{A}) \\
& = & (\<\pi_0, !u\> \times 1)\<\<\pi_0,\pi_2\>\action{A},\<\pi_1,\pi_2\>\action{A}\> \mu \mbox{ (by Corollary \ref{corMuLinearity})} \\
& = & \<\<\pi_0,\pi_1\>\action{A}, \<!u,\pi_1\>\action{A}\>\mu \\
& = & \<\action{A},\pi_1\>\mu \mbox{ (by Lemma \ref{lemmaBundleActionUnital})} 
\end{eqnarray*}
as required.
\end{proof}

One may wonder why the action $\action{A}$ wasn't defined directly as the solution to this system (after all, its vector field is linear). However, proving the required results using this as the definition seemed difficult; we could not see a way to prove even the unitality of the action if it was defined in this way.  

\begin{corollary}\label{cor0DerivOfAction}
The time derivative at 0 of the map $\action{A}: C \times A \to A$ is the lift map $\lambda: A \to TA$.
\end{corollary}
\begin{proof}
By the previous result, the time derivative at 0 of $\action{A}$, that is, the map
	\[ \<!c_0,1\>(c_1 \times 0)T(\action{A}) \]
is equal to $\<q\zeta,1\>\mu$. However, by definition of $\mu$ and naturality of $0$,
	\[ \<q\zeta,1\>\mu = \<q\zeta,1\>(0 \times \lambda)T(\sigma) = \<q 0 T(\zeta), \lambda\>T(\sigma) = \lambda. \]
\end{proof}

\begin{remark}
In smooth manifolds, the above is typically how the lift map $\lambda: A \to TA$ for a vector bundle is \emph{defined}; that is, it is defined as the derivative at 0 of the map $\mathbb{R} \times A \to A$ which sends $(t, a)$ to $t \cdot a$ (for example, see \cite[pg. 55]{natural}).  
\end{remark}

The following is a useful universal property of a differential curve object; we will use it to prove associativity of each $\action{A}$.  

\begin{lemma}\label{unitObject}
Suppose that $q: A \to M$ and $q': B \to M'$ are differential bundles, and $(f,g)$ is a bundle morphism between them. Then there exists a unique map 
\[
	\xymatrix{ A \ar[r]^{f} \ar[d]_{\<!u,1\>} & B \\ C \times A \ar@{.>}[ur]_{\hat{f}} & } 
\]
such that $\hat{f}$ is linear in $C$ (that is, $(\hat{f},fq')$ is a linear bundle morphism from $\pi_1: C \times A \to A$ to $q': B \to M'$).  
\end{lemma}
\begin{proof}
Suppose we have such an $f$. If we define $\hat{f}$ as the composite
	\[ C \times A \to^{1 \times f} C \times B \to^{\action{B}} B \]
Then the diagram commutes since $\action{B}$ is unital with unit $u$ (Lemma \ref{lemmaBundleActionUnital}), and $\hat{f}$ is linear in $C$ since $\action{B}$ is (Proposition \ref{propActionLinearInC}).  

For uniqueness, suppose we have a $\hat{f}: C \times A \to B$ with those properties. We then claim that both $\<(1 \times f)\action{B},\pi_1f\>$ and $\<\hat{f},\pi_1f\>$ solve the system $(B_2,\<\mu,\pi_10\>,\<q \zeta',f\>)$. First, by Proposition \ref{propActionSolution}, $\<\action{B},\pi_1\>$ solves the system $(B_2,\<\mu,\pi_10\>,\<q'\zeta',1\>)$. Thus, by Lemma \ref{lemmaParamFromFlow}, $(1 \times f)\<\action{B},\pi_1\> = \<(1 \times f)\action{B},\pi_1f\>$ solves the system $(B_2,\<\mu,\pi_10\>,f\<q'\zeta',1\>) =(B_2, \<\mu,\pi_10\>,\<q\zeta',f\>)$. 

To show that $\<\hat{f},\pi_1f\>$ is also a solution of this system, we need to show that the following diagram commutes:
\[
\bfig
	\square<750,350>[C \times A`T(C \times A)`B_2`T(B_2);c_1 \times 0`\<\hat{f},\pi_1f\>`T(\<\hat{f},\pi_1f\>)`\<\mu,\pi_10\>]
	\morphism(-750,350)<750,0>[A`C \times A;\<!c_0,1\>]
	\morphism(-750,350)|b|<750,-350>[A`B_2;\<q \zeta', f\>]
\efig
\]
The initial condition follows since $\hat{f}$ is linear in $C$. For the derivative condition, consider
\begin{eqnarray*}
&   & (c_1 \times 0)T(\hat{f}) \\
& = & (\<1,!u\>\mu \times 0)T(\hat{f}) \mbox{ (by Lemma \ref{lemma_U} and Remark \ref{rmkMuOnDiffObject})} \\
& = & (\<\pi_0,!u\> \times 1)(\mu \times 0)T(\hat{f}) \\
& = & (\<\pi_0,!u\> \times 1)(\<\<\pi_0,\pi_2\>\hat{f},\<\pi_1,\pi_2\>\hat{f})\mu \mbox{ (since $\hat{f}$ is linear in $C$, and using \cite[Lemma 2.17]{diffBundles})} \\
& = & \<\<\pi_0,\pi_1\>\hat{f},\<!u,\pi_1\>\hat{f}\> \mu \\
& = & \<\hat{f},\pi_1f\>\mu \mbox{ (by assumption on $\hat{f}$)}
\end{eqnarray*}
as required.  

Thus, $\<(1 \times f)\action{B},\pi_1f\>$ and $\<\hat{f},\pi_1f\>$ solve the same system and so are equal, and thus by equality of components, $(1 \times f)\action{B} = \hat{f}$, as required.  
\end{proof}

\begin{corollary}\label{corActionAssoc}
$\action{A}$ is an associative action of $C$ on $A$: that is, the following diagram commutes:
\[
\bfig
	\square<750,350>[C \times C \times A`C \times A`C \times A`A;1 \times \action{A}`\action{C} \times 1`\action{A}`\action{A}]
\efig
\]
\end{corollary}
\begin{proof}
By unitality of $\action{A}$, the following diagram commutes:
\[
\xymatrix{ C \times A \ar[r]^{\action{A}} \ar[d]_{\<!u,1\>} & A \\ C \times C \times A \ar[ur]_{(1 \times \action{A})\action{A}} }
\]
and $(1 \times \action{A})\action{A}$ is linear in $C$ since $\action{A}$ is. Moreover, by unitality of $\action{C}$, the following diagram also commutes:
\[
\xymatrix{ C \times A \ar[r]^{\action{A}} \ar[d]_{\<!u,1\>} & A \\ C \times C \times A \ar[ur]_{(\action{C} \times 1)\action{A}} }
\]
and is linear in $C$ since $\action{C}$ and $\action{A}$ are. Thus, by uniqueness of such maps (Lemma \ref{unitObject}), 
	\[ (1 \times \action{A})\action{A} = (\action{C} \times 1)\action{A}. \]
\end{proof}

We can now prove our main result.  

\begin{theorem}\label{thmActions}
If $C$ is a differential curve object, then $(C,\action{C},u)$ is a commutative differential rig, and any differential bundle $q: A \to M$ is a differential $C$-module with action $\action{A}$.
\end{theorem}
\begin{proof}
By \ref{propActionLinearInA}, \ref{propActionLinearInC}, \ref{lemmaBundleActionUnital}, and \ref{corActionAssoc}, $\action{C}: C \times C \to C$ and each $\action{A}$ is bilinear, (left) unital, and associative.  

Thus, the only thing left to prove is that $\action{C}$ is commutative. For this, it will first be helpful to write the map $\action{C}: C \times C \to C$ in a different way. Recall that for a differential object, the $\{ \ \}$ operation is equivalent to post-composing by the map $\hat{p}$. Moreover, applying the isomorphism $TC \cong C \times C$, the map $C \times C \to^{(0 \times \lambda)} T(C \times C)$ is the map
	\[ C \times C \to^{\<0,\pi_1,\pi_0,0\>} C \times C \times C \times C. \]
Combining these results gives
	\[ \action{C} = \<0,\pi_1,\pi_0,0\>T(\expf{C})\hat{p} = \<0,\pi_1,\pi_0,0\>D(\expf{C}). \]
Thus, we have
\begin{eqnarray*}
&  & \<\pi_1,\pi_0\>\action{C} \\
& = & \<\pi_1,\pi_0\>\<0,\pi_1,\pi_0,0\>D(\expf{C}) \\
& = & \<0,\pi_0,\pi_1,0\>D^2(e) \mbox{ (by \ref{prop_eProperties}.vi)} \\
& = & \<0,\pi_1,\pi_0,0\>D^2(e) \mbox{ (by CD.7)} \\
& = & \<0,\pi_1,\pi_0,0\>D(\expf{C}) \mbox{ (by \ref{prop_eProperties}.vi)} \\
& = & \action{C}  
\end{eqnarray*}
So that indeed $\action{C}$ is commutative.  
\end{proof}

Using Lemma \ref{unitObject}, we can also show that linear maps between differential bundles are precisely those maps that preserve these actions.

\begin{proposition}\label{propPresOfActions}
If $A$ and $B$ are differential bundles, and $f: A \to B$ is a bundle morphism between them, then $f$ is linear if and only if $f$ preserves the actions associated to $A$ and $B$; that is, if and only if the following diagram commutes:
\[
\bfig
	\square<500,350>[C \times A`A`C \times B`B;\action{A}`1 \times f`f`\action{B}]
\efig
\]
\end{proposition}
\begin{proof}
Assume that $f$ is linear.  Pre-composing each of the composites in the square above by $\<!u,1\>: A \to C \times A$, we get
	\[ \<!u,1\>(1 \times f)\action{B} = f \<!u,1\>\action{B} = f \]
and
	\[ \<!u,1\>\action{A}f = f \]
since $u$ is a unit for these actions.  Moreover, $(1 \times f)\action{B}$ is linear in $C$ since $\action{B}$ is, and $\action{A}f$ is linear in $C$ since $\action{A}$ is linear in $C$ and we have assumed $f: A \to B$ is linear. Thus, by Lemma \ref{unitObject}, $(1 \times f)\action{B} = \action{A}f$.

On the other hand, suppose that the above diagram commutes. Then their time derivatives at $0$ are equal; that is,
	\[ \<!c_0,1\>(c_1 \times 0)T(\action{A}f) = \<!c_0,1\>(c_1 \times 0)T((1 \times f)\action{B}). \]
However, by Corollary \ref{cor0DerivOfAction}, $\<!c_0,1\>(c_1 \times 0)T(\action{A}) = \lambda_A$, so the left side equals $\lambda_A T(f)$. Moreover, by naturality of $0$,
	\[ \<!c_0,1\>(c_1 \times 0)T((1 \times f)\action{B}) = f\<!c_0,1\>(c_1 \times 0)T(\action{B}) \]
which again by Corollary \ref{cor0DerivOfAction} equals $f \lambda_B$. 
Thus $\lambda_A T(f) = f \lambda_B$, so $f$ is linear.  
\end{proof}

\subsection{Notes on the linear completeness axiom}\label{section:axiomSimplification}

In this final section, we consider two aspects of the linear completeness axiom itself. First, we show that in many cases checking the linear completeness axiom can be simplified. Then, we consider the issue of when $\VF$ inherits linear completeness.  

In many cases, the linear completeness axiom can be reduced to checking that the result holds when $V^M = c_1 \times 0$.   

\begin{proposition}
Assume that all pullbacks along differential bundles exist and are preserved by each $T^n$. Suppose also that $(C,c_0,c_1)$ has the following property: for any differential bundle $q: A \to C \times E$, any linear vector field of the form $(V^A, c_1 \times 0)$ has a solution. Then $(C, c_0,c_1)$ satisfies linear completeness.  
\end{proposition}
\begin{proof}
Suppose we are given the assumptions in the linear completeness axiom: that is, we have a differential bundle $q: E \to M$ and a linear vector field $(V^E,V^M)$ on it so that $V^M$ is complete. We need to show that $V^E$ is complete.

By Corollary \ref{corAllSolutions}, there is a solution to any system with $V^M$ as the vector field.  Let $\gamma^M: C \times E \to M$ be a solution to $(M,V^M,q)$. Let $P$ be the pullback of $q$ along $\gamma^M$, with projections $\pi_0$ and $\pi_1$:
\[
\bfig
	\square<500,350>[P`E`C \times E`M;\pi_1`\pi_0`q`\gamma^M]
\efig
\]
By \cite[Lemma 2.7]{diffBundles}, $\pi_0: P \to C \times E$ is a differential bundle. We will build a linear vector field over $c_1 \times 0$ on this differential bundle.  

By assumption $T$ of the diagram above is also a pullback. Thus, we can define a map $V^P: P \to TP$ by the pairing
	\[ V^P := \<\pi_0(c_1 \times 0), \pi_1 V^E\> = (c_1 \times 0) \times V^E. \]
This is a well-defined pairing into $TP$ since
\begin{eqnarray*}
&  & \pi_0 (c_1 \times 0)T(\gamma^M) \\
& = & \pi_0 \gamma^M V^M \mbox{ (by definition of $\gamma^M$)} \\
& = & \pi_1 q V^M \mbox{ (by pullback diagram for $P$)} \\
& = & \pi_1 V^E T(q) \mbox{ (since $V^E$ is over $V^M$)} 
\end{eqnarray*}
Naturality of $p$ shows that $V^P$ is a vector field, it is clearly over $c_1 \times 0$, and it is linear since $V^E$ is. Thus, by assumption, $V^P$ is a complete vector field, so that by Corollary \ref{corAllSolutions}, any system with $V^P$ as a vector field has a solution.  

Next, define a map $i: E \to P$ by the pairing
	\[ E \to^{\<\<!c_0,1\>,1\>} P \]
This is a well-defined pairing since by definition of $\gamma^M$, $\<!c_0,1\>\gamma^M = q$.  

Now define $\gamma^P$ as a solution to the system $(P,V^P,i)$. Then consider the following diagram:
\[
\bfig
	\square<650,350>[P`TP`E`TE;V^P`\pi_1`T(\pi_1)`V^E]
	\square(0,350)<650,350>[C \times E`T(C \times E)`P`TP;c_1 \times 0`\gamma^P`T(\gamma^P)`]
	\morphism(-700,700)|b|<700,-700>[E`E;1_E]
	\morphism(-700,700)<700,0>[E`C \times E;\<!c_0,1\>]
	\morphism(-700,700)<700,-350>[E`P;i]
\efig
\]
The top regions commute by definition of $\gamma^P$, the bottom left triangle by definition of $i$, and the bottom right square by definition of $V^P$. Thus, the composite $\gamma^P \pi_1: C \times E \to E$ solves $V^E$, as required.
\end{proof}

As a final note, we consider the linear completeness axiom for the tangent categories $\VF$ and $\FLOW$. As noted in Proposition \ref{propVFCurveObject}, if $\X$ has a curve object $C = (C,c_1,c_0)$, then $\bar{C} = ((C,0_C),c_1,c_0)$ is a curve object in $\VF$ (and $((C,\pi_1),c_1,c_0)$ is a curve object in \FLOW). However, it is not immediate that if $C$ satisfies the linear completeness axiom then $\bar{C}$ does as well. In particular, the issue is determining the differential bundles in $\VF$. In Proposition \ref{propDiffBundlesInVF}, we saw that every linear vector field gives rise to a differential bundle in $\VF$, and one can show (see below) that the linear completeness axiom does hold for such differential bundles. However, not all differential bundles in $\VF$ need be of that form.  

One way to resolve this problem is to consider the linear completeness axiom relative to a class of ``endemic'' fibre products. This notion was first considered in \cite[Definition 6.1]{affine}); here we briefly recall it:

\begin{definition}
A \textbf{basic fibre product} in a tangent category $\X$ is either a fibre product of the form $T_nM$, or a pullback diagram witnessing the universality of the lift. A \textbf{class of endemic fibre products in $\X$} is a class ${\cal F}$ of finite fibre product diagrams in $\X$ that contains all basic fibre products and is closed under application of the tangent functor $T$.  
\end{definition}

If $\X$ has a class of endemic class of fibre products ${\cal F}$, then there is a natural choice of endemic fibre products in $\VF$: the class ${\cal F}_{V}$ of fibre products in $\VF$ for which application of $U$ produces a fibre product in ${\cal F}$; similarly there is a natural choice of endemic fibre products in $\FLOW$.  

\begin{definition}
In a tangent category with endemic fibre products $\X$, an \textbf{endemic differential bundle} is a differential bundle $q: A \to M$ such that its pullback powers $A_n$ and the diagram witnessing the universality of its lift are all in ${\cal F}$.  
\end{definition}

We can then characterize the endemic differential bundles in $\VF$:

\begin{proposition}\label{propEndemicDiffBundlesInVF}
If $\X$ has a class of endemic fibre products ${\cal F}$, then an endemic differential bundle in $\VF$ (relative to the class ${\cal F}_{V}$) is precisely an endemic differential bundle $q: A \to M$ in $\X$, together with a linear vector field $(V^A,V^M)$ on it.  
\end{proposition}
\begin{proof}
We saw earlier (Proposition \ref{propDiffBundlesInVF}) that a differential bundle with a linear vector field on it gives a differential bundle in $\VF$; by definition if the differential bundle is endemic in $\X$ then this an endemic differential bundle in $\VF$. Conversely, if $(q: (A,V^A) \to (M,V^M),\lambda)$ is an endemic differential bundle in $\VF$, then as in the proof of Proposition \ref{propDiffBundlesInVF}, $q$ being a vector field morphism means that $V^A$ be over $V^M$ and $\lambda$ being a vector field morphism means $V^A$ is a linear vector field; the assumption that the differential bundle is endemic (with respect to the class ${\cal F}_V$) means that $(q: A \to M,\lambda)$ must be a differential bundle in $\X$.  
\end{proof}

\begin{definition}
Suppose that $\X$ has a class of endemic fibre products ${\cal F}$, and has a curve object $C = (C,c_1,c_0)$. Say that $C$ \textbf{has linear completeness relative to ${\cal F}$} if $C$ satisfies the requirements of the linear completeness axiom for endemic differential bundles (but not necessarily for all differential bundles).
\end{definition}


Relative to their canonical choice of endemic fibre products, $\VF$ and $\FLOW$ satisfy linear completeness:

\begin{proposition}
Suppose that $\X$ has a class of endemic fibre products ${\cal F}$ and a curve object $(C,c_1,c_0)$ satisfying linear completeness relative to ${\cal F}$. Then $((C,0_C),c_1,c_0)$ satisfies linear completeness with respect to ${\cal F}_{V}$ in $\VF$, and $((C,\pi_1),c_1,c_0)$ satisfies linear completeness with respect to the canonical class of endemic fibre products in $\FLOW$.  
\end{proposition}
\begin{proof}
 Suppose $(q: (A,V^A) \to (M,V^M),\lambda)$ is an endemic differential bundle in $\VF$, and $(W^A,W^M)$ is a linear vector field on $(A,V^A)$. Then by Proposition \ref{propEndemicDiffBundlesInVF}, $(q: A \to M,\lambda)$ is an endemic differential bundle in $\X$ and $(W^A,W^M)$ is a linear vector field on $A$. Thus, by assumption on $C$, $W^A$ has a solution, and so by Proposition \ref{propVFCurveObject}, this solution is also a solution in $\VF$. By isomorphism of categories, a similar result holds for $\FLOW$.  
\end{proof}

\section{Conclusion and future results}\label{secConclusions}

To our knowledge, curve objects in settings for differential geometry have not previously been isolated as an abstract structure worthy of explicit study.  The presence of a curve object in a tangent category allows many fundamental ideas from differential geometry to have a direct expression. As we have discussed, this approach to these ideas provides new structural perspectives on classical differential geometric ideas. However, it would be remiss of us not to point out that there is much more to this subject than has been covered here:
\begin{itemize}
	\item In this paper, we essentially restricted our attention to complete vector fields: ones with solutions defined everywhere. However, in smooth manifolds, while an arbitrary vector field need not have a complete solution, all vector fields have partially defined solutions. To express these ideas, one can turn to restriction tangent categories \cite[Definition 6.14]{sman3}, which are tangent categories in which maps need only be partially defined. In such a setting, every dynamical system would have a (partial) solution. However, expressing the sense in which this solution is unique requires some care. 
	\item There is much more that can be done with solutions to linear systems. In particular, trigonometric and hyperbolic functions arise as solutions to certain linear systems. Thus, given a curve object with linear completeness, one can define abstract versions of hyperbolic sin and cos (and, with negatives, ordinary sin and cos) and then show that some of the fundamental properties of these functions result from general results about solutions of systems. 
	\item While we have discussed the Lie derivative of one vector field along another (that is, their Lie bracket), we have not given the more general definitions of Lie derivative of a vector field with respect to a differential form, etc.; thus, there is more that could be done here.  
	\item As noted in the introduction, tangent categories have not previously assumed any sort of ``real number object'' which acts on differential bundles and objects. However, we have seen that if $C$ is a differential curve object, then it does play this role. This may allow even more results and ideas from differential geometry (some of which rely on such actions) to be generalized to tangent categories.  
	\item A fundamental result about differential equations that we have not discussed is Frobenius's Theorem. Most proofs of Frobenius' theorem involve constructing a system of $n$ vector fields which pairwise commute (for example, see the proof of Theorem 14.1 in \cite{equivInvSym}). As noted in section \ref{secCommFlowsAndVFs}, such a system can be seen as a vector field in the tangent category $\sf{VF}^{n-1}(\X)$.  Thus, some of the perspectives offered in this paper may be useful in defining and proving Frobenius' theorem in this setting.  
	\item Recently, there has been an effort to give a categorical axiomatization of integration and anti-derivatives. The first and third authors introduced Cartesian integral categories \cite{cockett2018cartesian}, the integral analogue of Cartesian differential categories, whose integration combinator generalizes line integration of smooth functions over curves. By the fundamental theorems of calculus, the line integral or antiderivative of a smooth function is the solution to the differential equation $y^\prime = f(x)$. Therefore, one should study how one can obtain antiderivatives and integration from a curve object in a Cartesian differential category. This may open the door to answering what integration should be for tangent categories (which at this point is still very much a mystery). 
\end{itemize}

Thus, there is much more to be done with these ideas.  

\bibliography{curveObjectsTotalV2}

\begin{thebibliography}{10}

\bibitem{geoVectorFields}
Stephano Biagi and Andrea Bonfiglioi.
\newblock {\em An Introduction to the Geometrical Analysis of Vector Fields}.
\newblock World Scientific, 2019.

\bibitem{affine}
R.F. Blute, G.~S.~H. Cruttwell, and R.B.B. Lucyshyn-Wright.
\newblock Affine geometric spaces in tangent categories.
\newblock {\em Theory Appl. Categ.}, 34:405--437, 2019.

\bibitem{sman3}
J.~R.~B. Cockett and G.~S.~H. Cruttwell.
\newblock Differential structure, tangent structure, and {SDG}.
\newblock {\em Appl. Categ. Structures}, 22(2):331--417, 2014.

\bibitem{lieBracket}
J.~R.~B. Cockett and G.~S.~H. Cruttwell.
\newblock The {J}acobi identity for tangent categories.
\newblock {\em Cahiers de Topologie et Geom\'etrie Diff\'erential
  Cat\'egoriques}, 59(1):10--92, 2015.

\bibitem{connections}
J.~R.~B. Cockett and G.~S.~H. Cruttwell.
\newblock Connections in tangent categories.
\newblock {\em Theory and Applications of Categories}, 32(26):835--888, 2018.

\bibitem{diffBundles}
J.~R.~B. Cockett and G.~S.~H. Cruttwell.
\newblock Differential bundles and fibrations in tangent categories.
\newblock {\em Cahiers de Topologie et Geom\'etrie Diff\'erential
  Cat\'egoriques}, LIX(1):10--92, 2018.

\bibitem{cockett2018cartesian}
J.R.B. Cockett and J.-S.~P. Lemay.
\newblock Cartesian integral categories and contextual integral categories.
\newblock {\em Electronic Notes in Theoretical Computer Science}, 341:45--72,
  2018.

\bibitem{diffForms}
G.S.H. Cruttwell and R.B.B. Lucyshyn-Wright.
\newblock A simplicial foundation for differential and sector forms in tangent
  categories.
\newblock {\em Journal of Homotopy and Related Structure}, 13(4):867--925,
  2018.

\bibitem{hille}
Einar Hille.
\newblock {\em Ordinary differential equations in the complex domain}.
\newblock Dover, 1976.

\bibitem{bart-jacobs}
B.~Jacobs.
\newblock {\em Categorical logic and type theory}.
\newblock Elsevier, 1999.

\bibitem{kockReyesSolutionsODEs}
A.~Kock and G.~Reyes.
\newblock Ordinary differential equations and their exponentials.
\newblock {\em Central European Journal of Mathematics}, 4(1):64--81, 2006.

\bibitem{natural}
Ivan Kol{\'a}{\v{r}}, Peter~W. Michor, and Jan Slov{\'a}k.
\newblock {\em Natural operations in differential geometry}.
\newblock Springer-Verlag, Berlin, 1993.

\bibitem{convenient}
Andreas Kriegl and Peter~W. Michor.
\newblock {\em The convenient setting of global analysis}, volume~53 of {\em
  Mathematical Surveys and Monographs}.
\newblock American Mathematical Society, 1997.

\bibitem{langDiffGeometry}
Serge Lang.
\newblock {\em Fundamentals of Differential Geometry}.
\newblock Springer, 1999.

\bibitem{lee}
John~M. Lee.
\newblock {\em Introduction to smooth manifolds}.
\newblock Springer-Verlag, 2003.

\bibitem{lemayExponential}
J.-S.~P. Lemay.
\newblock Exponential functions in cartesian differential categories.
\newblock {\em Applied Categorical Structures}, 2020.

\bibitem{roryConnections}
R.B.B. Lucyshyn-Wright.
\newblock On the geometric notion of connection and its expression in tangent
  categories.
\newblock {\em Theory Appl. Categ.}, 33:832--866, 2018.

\bibitem{equivInvSym}
Peter Olver.
\newblock {\em Equivalence, Invariants, and Symmetry}.
\newblock Cambridge University Press, 1995.

\bibitem{spivak2}
M.~Spivak.
\newblock {\em A comprehensive Introduction to Differential Geometry, Volume
  Two}.
\newblock Publish or Perish, Inc., 3rd edition, 1999.

\end{thebibliography}

\end{document}